\documentclass[11pt, reqno]{amsart}

\usepackage[french, english]{babel}
\usepackage[normalem]{ ulem }
\usepackage{soul}

\usepackage{cite}
\usepackage[latin1]{inputenc}
\usepackage{amssymb}
\usepackage{amsmath}
\usepackage{a4wide}
\usepackage{eucal}
\usepackage[]{epsfig}
\usepackage{wrapfig}
\usepackage[]{pstricks}
\usepackage{xcolor}
\usepackage{flafter}
\usepackage{amscd,amsthm}
\usepackage{amsfonts}
\usepackage{graphicx}
\usepackage{hyperref}

\usepackage{verbatim}

\usepackage{float}

\usepackage{enumerate, xspace}

\usepackage[charter]{mathdesign}

\usepackage{tikz}
\usetikzlibrary{arrows,decorations.pathmorphing,backgrounds,positioning,fit,petri,cd}

\usepackage{xcolor}
\usetikzlibrary{arrows.meta}
\usepackage{pstricks,pst-node}
\usetikzlibrary{backgrounds}
\usetikzlibrary{patterns,fadings}
\usetikzlibrary{arrows,decorations.pathmorphing}
\usetikzlibrary{shapes}
\usetikzlibrary{calc}
\definecolor{OliveGreen}{rgb}{0, 0.26, 0.15}
\definecolor{dark red}{rgb}{0.5, 0.05, 0.13}
\definecolor{dark red 2}{rgb}{0.81, 0.09, 0.13}
\definecolor{dark blue}{rgb}{0, 0.18, 0.39}
\definecolor{dark blue 2}{rgb}{0.03, 0.27, 0.49}
\definecolor{dark green2}{rgb}{0.07, 0.53, 0.03}
\definecolor{dark green}{rgb}{0, 0.44, 0}

\newtheorem{thm}{Theorem}[section]
\newtheorem{lem}[thm]{Lemma}

\newtheorem{prop}[thm]{Proposition}

\newtheorem{definition}[thm]{Definition}

\newtheorem{rem}[thm]{Remark}

\newcommand\PP{{\mathbb P}}

\newcommand\RR{{\mathbb R}}

\newcommand\ZZ{{\mathbb Z}}

\definecolor{lightcarminepink}{rgb}{0.9, 0.4, 0.38}
\definecolor{lightcoral}{rgb}{0.94, 0.5, 0.5}
\definecolor{lightcornflowerblue}{rgb}{0.6, 0.81, 0.93}
\definecolor{lightcyan}{rgb}{0.88, 1.0, 1.0}
\definecolor{lavenderblush}{rgb}{1.0, 0.94, 0.96}
\definecolor{deeppeach}{rgb}{1.0, 0.8, 0.64}
\definecolor{darkchampagne}{rgb}{0.76, 0.7, 0.5}
\definecolor{desertsand}{rgb}{0.93, 0.79, 0.69}
\definecolor{classicrose}{rgb}{0.98, 0.8, 0.91}
\definecolor{myyeallow}{rgb}{0.98, 0.91, 0.71}
\definecolor{carolinablue}{rgb}{0.6, 0.73, 0.89}
\definecolor{antiquewhite}{rgb}{0.98, 0.92, 0.84}
\definecolor{mycolor}{rgb}{0.98, 0.91, 0.71}
\usepackage{enumerate}

\definecolor{camel}{rgb}{0.76, 0.6, 0.42}
\definecolor{brightcerulean}{rgb}{0.11, 0.67, 0.84}
\definecolor{cerulean}{rgb}{0.0, 0.48, 0.65}
\definecolor{Gray}{rgb}{0.5, 0.5, 0.5}

\numberwithin{equation}{section}

\usepackage{dsfont}
\usepackage{ulem}
\usepackage{setspace}

\usetikzlibrary{shapes.misc,arrows,decorations.pathmorphing,backgrounds,positioning,fit,petri,shapes}

\title[]{Diffusive to super-diffusive behavior in boundary driven exclusion }
\author{Patr\'icia Gon\c calves and Stefano Scotta}

\begin{document}

\maketitle{}

\begin{abstract}The purpose of this article is to study the hydrodynamic limit of the symmetric exclusion process with long jumps and in contact with infinitely extended reservoirs for a particular critical regime. The jumps are given in terms of a transition probability that has finite or infinite variance and the hydrodynamic equation is a diffusive equation, in the former case, or a fractional equation, in the latter case.  In this work, we treat the critical case, that is, when the variance  grows logarithmically with	the size of the system and corresponds to the case in which there is a transition from diffusive to super-diffusive behavior.
\end{abstract}

\section{Introduction}
\label{sec:1}

Since the pioneering work of Spitzer in \cite{Spitzer},  interacting particle systems were introduced in the mathematics community and have been extensively studied. These systems consist of microscopic models of random particles whose dynamics conserve one quantity (or more). Under a suitable space/time scale, the law of large numbers for the conserved quantity of the system can be obtained and this is known in the literature as the hydrodynamic limit. This limit states that the particles' random dynamics conserve a quantity whose space/time evolution can be well approximated by a solution of a partial differential equation, the \textit{hydrodynamic equation} \cite{KL}.

 One of the most classical interacting particle systems is the exclusion process. The exclusion rule dictates that jumps can occur, at a certain rate, if and only if the destination site is empty, otherwise the jump is suppressed.  In this article, we focus on the exclusion process evolving on a discrete space, namely, a finite lattice with $N$ points, where $N$ is a scaling parameter that will be taken to infinity. We consider a transition probability $p(\cdot)$, that regulates the intensity of the size of the jump, allowing long jumps and symmetric $p(z)=p(-z)$.  The model under investigation in this article was previously studied in  \cite{MJ} when evolving in $\mathbb Z^d$. Here, we restrict the dynamics of \cite{MJ} to a finite lattice but we place it in contact with stochastic reservoirs, that can inject or remove particles everywhere in the discrete space.  We observe that the reservoirs' action can either change the nature of the equation or bring additional boundary conditions to it.  Before introducing the model in detail, we recall previous results about its hydrodynamic behavior that are summarized in  Figure \ref{figura}.

First, we observe that it was previously studied in \cite{Adriana} a simpler version of our model, i.e. when the transition probability only allows nearest-neighbor jumps and the reservoirs have an impact only at boundary points. There, the hydrodynamical behavior is given by the heat equation with different boundary conditions. These boundary conditions depend on a parameter that regulates the intensity of the reservoirs' dynamics. 
 Some years later, in \cite{BGJO} it was introduced the model under investigation in this article, with a transition probability that: depends only on the length of the jump,  decreases as the jump's size increases, and  has finite variance; and the reservoirs have an impact everywhere in the system. As in \cite{Adriana}, the hydrodynamic limit is given by the heat equation, however, there are two extra regimes when compared to \cite{Adriana}.  In the case where both the exclusion and the reservoirs' dynamics have the same strength, the hydrodynamics is given by a reaction-diffusion equation, while when the reservoirs' dynamics is stronger than the exclusion dynamics, the hydrodynamics is given by a reaction equation.

 Later in \cite{BJGO2,mio}, it was studied the model of \cite{BGJO} when the variance of $p(\cdot)$ is infinite. Differently from the finite variance, the hydrodynamic equation is given in terms of the regional fractional Laplacian, which is the generator of processes as the censored process  (see \cite{Bogdan}) and/or the reflected L\'evy flight process (see \cite{guan}). An interesting review about this fractional operator can be found in \cite{frac}. Contrarily to the usual Laplacian operator, this last operator is non-local and it has a similar definition to the usual fractional Laplacian but is restricted to a finite domain. 
 
In this article we deduce the hydrodynamic limit in the case connecting the diffusive behavior studied in \cite{BJGO2} and the super-diffusive behavior studied in \cite{mio,BGJO}. Now we explain in detail the parameters of the model.  Our  choice of the transition probability, is given on $x, y \in \ZZ$ by \begin{equation}\label{probgen}
p(x,y)=p(y-x)=c_{\gamma}|x-y|^{-(\gamma+1)}
\end{equation} 
and $p(0)=0$. When $\gamma>2$ (resp. $\gamma\leq 2$)  the transition probability has  finite (resp. infinite) variance.  Above $c_\gamma$ is a normalizing constant, turning  $p(\cdot)$ into a probability. 
For this choice of $p(\cdot)$, the higher the value of $\gamma$ the lower is the probability of having a long jump. Another parameter that has an impact at the macroscopic level is $\theta\in\mathbb R$ that rules the reservoirs' strength so that the higher the value of  $\theta$ the weaker is the reservoirs' action. Whatever is the regime of $\gamma$, when $\theta$ is small (but negative),  the hydrodynamic equation has an additional reaction term with respect to the case $\theta\geq 0$, which comes from the fact that the reservoirs act in the whole discrete space where particles evolve. Moreover,  when  $\theta$ is very small (negative) so that the reservoirs' action becomes strong,  we obtain a reaction equation.  On the other hand, as the value of $\theta$ increases (positive), both in the case $1<\gamma<2$ {(\cite{BJGO2})} or $\gamma>2$ {(\cite{BGJO})},  the boundary conditions of the equation change. Indeed, in both regimes, the boundary conditions go from Dirichlet to Robin (or fractional Robin) and then  Neumann (or fractional Neumann), and this depends on the value of $\theta$. The case $0<\gamma<1$ is different (it corresponds to the case when $p(\cdot)$ has infinite first moment), since one crosses from a  reaction equation to a  regional fractional diffusion equation with Neumann boundary conditions, and this is a consequence of the particular properties of the regional fractional Laplacian operator on this regime of $\gamma$, for details see  \cite{mio,guan}. 
\begin{figure}[htb!]
	\begin{tikzpicture}[scale=0.28]
	\fill[color=red!15] (-20,-5) rectangle (-10,-18);
	\fill[color=red!15] (-10,-5) -- (5,-5) -- (20,-18) -- (-10,-18)-- cycle;
	\fill[color=green!5] (-10,-5) -- (5,-5) -- (5,5) -- cycle;
	\fill[color=green!25] (-20,-5) -- (-10,-5) -- (5,5) -- (5,12) -- (-20,12) -- cycle;
	\fill[color=blue!5] (5,-5) -- (20,-18) -- (20,5) -- (5,5) -- cycle;
	\fill[color=blue!20] (5,5) rectangle (20,12);
	\draw[-,=latex,dark red,ultra thick] (-20,-5) -- (5, -5) node[midway, below, sloped] {\scriptsize{{\textbf{\textcolor{dark red}{Frac. Reac. Diff. \& Dirichlet b.c. }}}}};
	\draw[dotted, ultra thick, white] (-20,10) -- (-20,-18);
	\draw[dotted, ultra thick, white] (-10,10) -- (-10,-18);
	\draw[-,=latex,dark red,ultra thick] (5,-5) -- (20, -18) node[midway, below, sloped] {\scriptsize{\textbf{{\textcolor{dark red}{Reac. Diffusion  \& Dirichlet b.c. }}}}};
	\draw[-,=latex,blue,ultra thick] (5,5) -- (20, 5) node[midway, sloped, below] {\scriptsize{{\textbf{{Diffusion \& Robin b.c.}}}}};
	\draw[-,=latex,OliveGreen ,ultra thick] (-10,-5) -- (5, 5) node[midway, sloped, below] {\scriptsize{{\textbf{\scriptsize{\textcolor{OliveGreen}{Fract. Diff. \& Robin b.c.}}}}}};
	\node[right, black] at (9,9) {\textbf{\scriptsize{\textcolor{dark blue}{Diffusion}}}};
	\node[right, black] at (7.7,7.7) {\textbf{\scriptsize{\textcolor{dark blue}{\& Neumann b.c.}}}};
	\node[right, white] at (9,-2) {\textbf{\scriptsize{\textcolor{dark blue 2}{Diffusion}}}} ;
	\node[right, white] at (7.7,-3.3) {\textbf{\scriptsize{\textcolor{dark blue 2}{\& Dirichlet b.c.}}}} ;
	\node[right, white] at (-18,7) {\textbf{\scriptsize{\textcolor{dark green}{Frac. Diff. \& Neumann b.c.}}}} ;
	\node[right, white] at (-3,-2.5) {\textbf{\scriptsize{\textcolor{dark green2}{Frac. Diff. }}}} ;
	\node[right, white] at (-4.5,-3.5) {\textbf{\scriptsize{\textcolor{dark green2}{\& Dirichlet b.c. }}}} ;
	\node[right, white] at (-12.4,-12) {\textbf{\scriptsize{\textcolor{red}{Reaction \& Dirichlet b.c.}}}} ;
	
	\node[rotate=270, above] at (-10,1) {\textcolor{white}{$\gamma = 1$}};
	\node[rotate=270, above] at (-25,1) {\textcolor{white}{$\gamma = 0$}};
	\node[right] at (20,5) {\textcolor{black}{$\theta=1$}};
	\node[left] at (-20,-5) {\textcolor{black}{$\theta=0$}};
	\draw[white,fill=white] (-10,-5) circle (1.5ex);
	
	\draw[<-,black] (11.5, -10.4) -- (13.6,-9) node[right] {\textcolor{black}{$\theta=2-\gamma$}};
	\draw[<-,black] (-3.3, -0.1) -- (-5,3) node[above] {\textcolor{black}{$\theta=\gamma-1$}};
	\node[rotate=270, above] at (4.1,13.8) {\textcolor{black}{$\gamma = 2$}};
	\draw[black,fill=black] (5,5) circle (1.8ex);
	\draw[black,fill=black] (5,-5) circle (1.8ex);
	\draw[ultra thick, black] (5,-5) -- (5,-18) node[midway, sloped, above] {\scriptsize{{{\scriptsize{Reaction}}}}} node[midway, sloped, below] {\scriptsize{{{\scriptsize{Dirichlet b.c.}}}}};;
	\draw[ultra thick, black] (5,5) -- (5,-5) node[midway, sloped, above] {\scriptsize{{{\scriptsize{Diffusion}}}}} node[midway, sloped, below] {\scriptsize{{{\scriptsize{Dirichlet b.c.}}}}};;
	\draw[ultra thick, black] (5,12) -- (5,5) node[midway, sloped, above] {\scriptsize{{{\scriptsize{Diffusion}}}}} node[midway, sloped, below] {\scriptsize{{{\scriptsize{Neumann b.c.}}}}};;
	\draw[<-,black] (4.7, 5.2) -- (-4,8.5) node[draw, above] {\scriptsize{\textcolor{black}{Diffusion \& Robin b.c.}}};
	\draw[<-,black] (4.7, -5.3) -- (-6,-9) node[draw, below] {\scriptsize{\textcolor{black}{Reac. Diffusion \& Dirichelt b.c.}}};
	\end{tikzpicture}
	\caption{Hydrodynamics depending on  $\theta$ (vertical axis) and $\gamma$ (horizontal axis).} 
	\label{figura}
\end{figure}
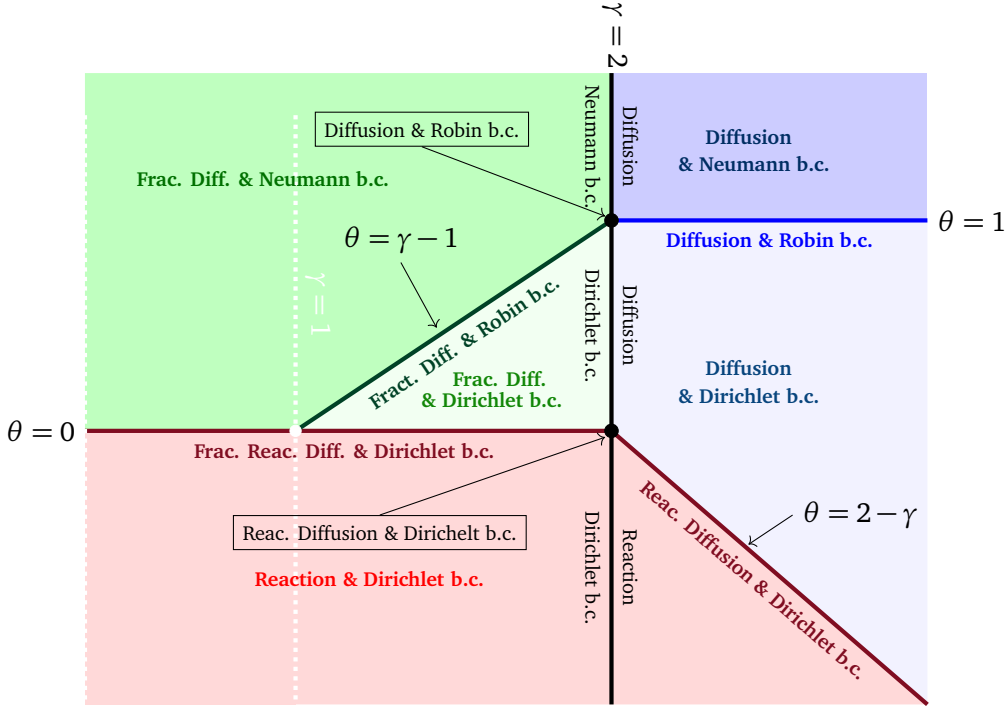
Observe that the macroscopic behavior of the system depends on the exponent $\gamma>0$ and the parameter $\theta$. In \cite{BGJO,BJGO2, mio} the  case $\gamma=2$ was left for exploration and it corresponds to the transition from the regime when $p(\cdot)$ has a finite variance to infinite variance, i.e. from diffusive to super-diffusive behavior. For this regime $\gamma=2$, $p(\cdot)$ has  variance of order $\log(N)$.  In this article we prove the hydrodynamic limit for all the regimes of $\theta$, see the black line in the figure below and we complete the hydrodynamic scenario for this model. We note that the nature of the hydrodynamic equation is still diffusive as in the case of finite variance, i.e. it is the heat equation for $\theta>0,$ a reaction-diffusion equation for $\theta=0$ and a reaction equation for $\theta<0$. As in the case of finite variance, for $\theta<1$ we get Dirichlet boundary conditions, for $\theta=1$ we get linear Robin boundary conditions, and for $\theta>1 $ we get Neumann boundary conditions.

Now, we highlight two important issues which complicate the analysis of this critical regime. The first one has to do with the time scale. For  $\gamma>2$ in the diffusive time scale $N^2$ we have a diffusive equation; for $\gamma<2$ in the sub-diffusive time scale $N^\gamma$ we have a fractional diffusion equation;  but for $\gamma=2$, since the variance of $p(\cdot)$ {is of order $\log(N)$ we need to take the time scale $N^2/\log(N)$}. This correction in the time scale allows controlling the logarithmic divergence of the variance but brings additional complications to derive the boundary conditions.  The second issue   is the fact that in the cases $\theta\in\{0,1\}$, the reservoirs' strength has to be scaled in a  precise way in order to get the reaction-diffusion equation with Dirichlet boundary conditions or the heat equation with Robin boundary conditions (see remark \ref{kappan}). 

A key step in the proof of the hydrodynamic limit is showing that the discrete operator arising from the microscopic particle dynamics, converges, in some sense, to the respective continuous operators. When $\gamma\leq 2$ the limit is the standard Laplacian operator but when  $\gamma<2$ it is the regional fractional Laplacian. This result was proved in Lemma 3.2 of \cite{BGJO} by extending the discrete operator to the whole set $\mathbb{Z}$ and deriving its limit to the usual Laplacian operator.  Here we managed to show this convergence without extending the discrete operator (Lemma \ref{L1conv}) by properly controlling boundary terms.

Once the hydrodynamic limit is proved, a natural question that follows is related to the description of the fluctuations around the hydrodynamical profile. This is a sort of central limit theorem for the empirical measure associated with this process. For our model, the equilibrium fluctuations were studied in \cite{fluc} when $\gamma>2$ but it would be very interesting to understand the case $\gamma \in (0,2]$, which is much more challenging. We leave this to future work.\\

{\textbf{Outline:}} This  paper is structured as follows. In Section \ref{model} we present in detail our model. In Section \ref{SR} we introduce all the notations that we need throughout the work and we present our result.  {Sections \ref{3} and \ref{2} are devoted to the proof of the hydrodynamic limit, namely, the characterization of limit points and the proof of tightness, respectively. In Section \ref{TEC} we prove some useful technical lemmas that are needed along  the proof of our result.  Finally in  Appendix \ref{appendix}  we discuss the uniqueness of the weak solutions that we obtain in the hydrodynamic limit. } 
 
\section{The model}\label{model}

We recall the  model introduced  in \cite{BGJO} that consists of  an exclusion process with  long jumps given in terms of a  symmetric transition probability, defined on $x, y \in \ZZ$ {by \eqref{probgen}. We take $\gamma=2$, so that} \begin{equation}\label{eq:prob_p}
p(x,y)=p(y-x)=\frac{c_{2}}{|x-y|^3}\mathbb{1}_{x\neq y}
\end{equation} 
 and the normalizing constant $c_2$ is defined by $c_2:=\big(\sum_{z \in \ZZ}|z|^{-3}\big)^{-1}.$ Before describing how the process evolves in time let us introduce here the following quantity that will be used through the article $$m:=\sum_{z\geq 1}zp(z).$$ Let us now explain the dynamics of the process. Fix $T>0$. For $N>1$, we consider particles evolving in the discrete set $\Lambda_N=\{1,\dots, N-1\}$  and we add infinitely many reservoirs at each site of $\mathbb{Z}\setminus \Lambda_N$.   {We consider} a  Markov process {$\{\eta_{t}\}_{t \in [0,T]}$} with state space $\Omega_N:=\{0,1\}^{\Lambda_N}$, and   for  $x \in \Lambda_N$, we say that the site $x$  is occupied (resp. empty) at time $t$ if $\eta_t(x)=1$ (resp.  $\eta_t(x)=0$).  Fix four parameters $\theta \in \mathbb{R}$, $\alpha, \beta \in [0,1]$ and $\kappa_N(\theta)$ defined in \eqref{kappa}. The dynamics can be described as follows:
\begin{itemize}
	\item to any couple of sites $(x,y)\in \Lambda_N$ we associate a Poisson process of rate $1$,  {according to which} we exchange the value of $\eta(x)$ and $\eta(y)$ with probability $p(x-y)$;
	\item to any couple $(x,y)$, with $x \in \Lambda_N$ and an integer $y\leq 0$, (resp. $y\geq N$) we associate a Poisson process of rate $1$,  and when there is an occurrence we flip the value of $\eta(x)$ to $1-\eta(x)$ with probability \begin{equation*}
	\begin{split}
	&\frac{\kappa_N(\theta)}{N^{\theta}}p(x-y)c_{x}(\eta;\alpha)\quad\Big(\textrm{resp.} \quad \frac{\kappa_N(\theta)}{N^{\theta}}p(x-y)c_{x}(\eta;\beta)\Big)\end{split}\end{equation*}
	where $\theta\in\mathbb R$ and  for {$\delta\in\{\alpha,\beta\}$},
	\begin{equation}\label{rate_c}
c_{x} (\eta;\delta) :=\left[ \eta(x) \left(1-\delta \right) + (1-\eta(x))\delta\right].
\end{equation}
\end{itemize}
The generator of the Markov process $\{\eta_t\}_{t\in [0,T]}$  {acts}  {on functions $f:\Omega_N \rightarrow \RR$  {as} $L_N f=L^0_N f+L_N^l f +L_N^r f $, where for $\eta\in\Omega_N$ } 
\begin{equation}\label{generators}
\begin{split}
&(L^0_N f)(\eta) =\cfrac{1}{2} \, \sum_{x,y \in \Lambda_N} p(x-y) [ f(\sigma^{x,y}\eta) -f(\eta)],\\
&(L_N^{l} f)(\eta) =\frac{\kappa_N(\theta)}{N^{\theta}}\sum_{\substack{x \in \Lambda_N\\ y \le 0}} p(x-y)c_{x}(\eta;\alpha) [f(\sigma^x\eta) - f(\eta)],\\
&(L_N^{r} f)(\eta)= \frac{\kappa_N(\theta)}{N^{\theta}}\sum_{\substack{x \in \Lambda_N \\ y \ge N}} p(x-y) c_{x}(\eta;\beta)  [f(\sigma^x\eta) - f(\eta)],
\end{split}
\end{equation}
and  \begin{equation}\label{kappa}
\kappa_N(\theta)=\begin{cases} \kappa, & \text{ if } \theta \notin\{0,1\};\\
\kappa \log(N), & \text{ if } \theta\in \{0,1\};
\end{cases}\end{equation}
for a fixed constant $\kappa>0$.
Above, 
\begin{equation*}
(\sigma^{x,y}\eta)(z) = 
\begin{cases}
\eta(z),& \textrm{if}\;\; z \ne x,y,\\
\eta(y),& \textrm{if}\;\; z=x,\\
\eta(x),& \textrm{if}\;\; z=y
\end{cases}
, \quad (\sigma^x\eta)(z)= 
\begin{cases}
\;\; \eta(z), &\textrm{if}\;\; z \ne x,\\
1-\eta(x),& \textrm{if}\;\; z=x.
\end{cases}
\end{equation*}
We are interested in analysing the space-time evolution of the  {empirical density of particles}. In order to have a non-trivial limit at the  macroscopic level, we need to accelerate the time of the  process by a factor $\Theta(N)$ (defined in \eqref{timescale}). Note that the infinitesimal generator of the  process $\{\eta_{t\Theta(N)}\}_{t\in[0,T]}=:\{\eta_t^N\}_{t\in [0,T]}$ is  $\Theta(N)L_N$.

\begin{rem}\label{kappan}
	We observe that the choice of $\kappa_N(\theta)$ above is justified by the following arguments. In order to have Robin boundary conditions for $\theta=1$ or the reaction term for $\theta=0$, we need to use the boundary strength as in the case   $\gamma>2$ analized in \cite{BGJO}. Since we scale time with a factor $\log(N)$, then it allows controlling the variance of $p(\cdot)$ and we get the  diffusive operator. Nevertheless, this  factor $\log(N)$ is suitable to control the Laplacian operator but it is not   {sufficient} to control gradients. In order to do so, we need to  consider the special choice of $\kappa_N(\theta)$ which permits to observe simultaneously the diffusive operator, the boundary conditions and the reaction term.
\end{rem}

\section{Statement of results}
\label{SR}
\subsection{Hydrodynamic equations}
\label{HE}

In order to define properly the notion of weak solutions of the several equations that we find, we need to introduce some notation. 

 {For any interval $I \subseteq \RR$, we denote by $C^k(I)$ (resp. $C_c^k (I)$)  the space of continuous real-valued functions (resp. with compact support included in $I$) with the first $k$-th derivatives being  continuous. Moreover, for $T>0$, we say that a function $H \in C^{m,n}([0,T]\times I)$ if $H(\cdot,x)\in C^m([0,T])$ and $H(t,\cdot) \in C^n(I)$ for any $x \in I$ and $t \in [0,T]$. Analogously $H \in C_c^{m,n}([0,T]\times I)$ if $H \in C^{m,n} ([0,T]\times I)$ and  $H(t,\cdot) \in C_c^n(I)$ for any $t \in [0,T]$.}
For any $d \in \mathbb{N}$ and any $G \in C^{m,n}([0,T]\times I)$, we denote by $G'$, $G''$ (or $\Delta G$), $G^{(d)}$, and $\partial_t G$ (or $\partial_s G$) resp. the first, the second, the $d$-th order derivative of $G$ with respect to the space variable, and the first derivative with respect to the time variable.  In addition, for any Polish space $E$, we consider the Skorokhod space $\mathcal{D}(I,E)$, i.e. the space of right continuous functions  with left limits, defined from  $I$ to $E$. Analogously, we denote by $C(I,E)$ the space of continuous functions from $I$ to $E$. 

The space $L^2(I)$  denotes the usual $L^2$ space with the Lebesgue measure, i.e. the space of functions $G$ such that $\int_IG(u)^2du <\infty$. This is the norm induced by the inner product $\langle G, H\rangle =\int_I G(u)H(u)du$ for $H, G \in L^2(I)$.  Analogously we denote by $L^1(I)$ the space of functions $G$ such that $\int_I|G(u)|du<\infty.$   {Let $\mathcal{H}^1(I)$ (resp. $\mathcal{H}_0^1 (I)$) be} the classical Sobolev space $\mathcal{W}^{1,2}(I)$ defined as the closure of $C^1(I)$ (resp. $C_c^1(I)$) with respect to the norm defined by 
$\|G\|^2_{\mathcal{H}^1(I)}:= \|G\|^2_{L^2(I)}+ \|G'\|^2_{L^2(I)}.$ 
Moreover, $L^2([0,T],\mathcal{H}^1(I))$ is the space of  functions $G$ defined on $[0,T]\times I$  for which  $\int_0^T\|G(s,\cdot)\|_{\mathcal H^1(I)}^2ds<\infty.$
If $I=[0,1]$ we will omit  the spacial variable in the notation and we just write $L^2$, $\mathcal{H}^1$, $\mathcal{H}_0^1$ and $L^1$. Let us introduce also,
 for  $u \in (0,1)$,
\begin{equation}\label{eq:V0V1}
V_0(u):=\alpha r^-(u)+\beta r^+(u) \quad \text{ and } \quad 	V_1(u):=r^-(u)+ r^+(u), 
\end{equation}
where
\begin{equation}
r^-(u):=\frac{c_2}{2u^2} \quad \text{ and } \quad 	r^+(u):=\frac{c_2}{2(1-u)^2}.
\end{equation}

To properly define the notion of weak solutions, we need to introduce other space of  test functions:
	\begin{equation}
	\begin{split}
		&\mathcal{S}_{Dir}:=\{G \in C^{1,2} ([0,T]\times (0,1)) : \; \; G_t(1)= G_t(0)=0, \forall  t \in [0,T]\};\\
	& {\mathcal{S}_{Rob}(\hat m, \hat c_2)}:=\{G \in C^{1,2} ([0,T]\times (0,1)) :\\& \quad \quad \quad \quad  \; \; \partial_uG_t(1)=\tfrac{\hat m}{  c_2}(\beta-G_t(1)) \text{ and } \partial_uG_t(0)=\tfrac{\hat m}{  c_2}(G_t(0)-\alpha), \forall  t \in [0,T]\}.
	\end{split}
	\end{equation}
 {Above $\hat m\geq 0$ and $\hat c_2>0$. Moreover, we use the notation  $\mathcal{S}_{Neu}:=\mathcal{S}_{Rob}(0, \hat c_2)$.}

\begin{definition}
	\label{reacdif}
	Let $\hat \kappa>0$, $\hat c_2 \geq 0$ and  let $g:[0,1]\rightarrow [0,1]$ be a measurable function. We say that  $\rho:[0,T]\times[0,1] \to [0,1]$ is a weak solution of the reaction-diffusion equation with inhomogeneous Dirichlet boundary conditions and initial condition $g$:
	\begin{equation}
	\label{eq:reacdif Equation}
	\begin{cases}
	&\partial_{t} \rho_{t}(u)= \hat c_2 \Delta \rho_t(u)+ \hat \kappa \big(V_0(u)-V_1(u)\rho_t(u)\big),  \quad (t,u) \in [0,T]\times(0,1),\\
	&\rho_t(0)=\alpha, \quad \rho_t(1)=\beta, \quad t \in (0,T]\\
	&{\rho}_{0}(u)= g(u),\quad u \in (0,1),
	\end{cases}
	\end{equation}
	if : 
	\begin{enumerate} 
		\item 
		If $\hat \kappa >0$, then  {$\int_0^T\int_0^1 \big\{(\frac{\alpha-\rho_s(u))^2}{u^\gamma}+\frac{(\beta-\rho_s(u))^2}{(1-u)^\gamma})\big\}duds<\infty.$}
		\item For all $t\in [0,T]$ and {all functions} $G \in C_c^{1,2} ([0,T]\times (0,1))$ we have that 
		\begin{equation}
		\label{eq:reac}
		\begin{split}
	{F_{{Reac}}(t, \rho,G,g)}:=\left\langle \rho_{t},  G_{t} \right\rangle& -\left\langle g,   G_{0}\right\rangle - \int_0^t\left\langle \rho_{s},\Big(\partial_s +  \hat c_2\Delta \Big) G_{s}  \right\rangle ds\\
		&+ \hat \kappa\int_0^t\int_0^1 G_s(u)\big(V_0(u)-V_1(u)\rho_s(u)\big)du ds=0.
		\end{split}   
		\end{equation}
		\item If $\hat c_2>0$ and $\hat \kappa=0$, $\rho_t(0)=\alpha$ and $\rho_t(1)=\beta$, $t$-almost  {everywhere} in $(0,T]$.
	\end{enumerate}
\end{definition}

\begin{definition}
	\label{dif}
	Let $\hat c_2 > 0$ and  let $g:[0,1]\rightarrow [0,1]$ be a measurable function. We say that  $\rho:[0,T]\times[0,1] \to [0,1]$ is a weak solution of the diffusion  equation with inhomogeneous Dirichlet boundary conditions and initial condition $g$:
	\begin{equation}
	\label{eq:dif Equation}
	\begin{cases}
	&\partial_{t} \rho_{t}(u)= \hat c_2 \Delta \rho_t(u),  \quad (t,u) \in [0,T]\times(0,1),\\
	&\rho_t(0)=\alpha, \quad \rho_t(1)=\beta, \quad t \in [0,T]\\
	&{\rho}_{0}(u)= g(u),\quad u \in (0,1),
	\end{cases}
	\end{equation}
	if : 
	\begin{enumerate} 
		\item For all $t\in [0,T]$ and {all functions} $G \in\mathcal{S}_{Dir}$ we have that 
		\begin{equation}
		\label{eq:dif}
		\begin{split}
		F_{{Dir}}(t, \rho,G,g):=\left\langle \rho_{t},  G_{t} \right\rangle& -\left\langle g,   G_{0}\right\rangle - \int_0^t\left\langle \rho_{s},\Big(\partial_s +  \hat c_2\Delta \Big) G_{s}  \right\rangle ds\\&-\int_0^t \big\{\beta\partial_uG_s(1)-\alpha \partial_u G(0)\big\}=0.
		\end{split}   
		\end{equation}
		\item  $\rho_t(0)=\alpha$ and $\rho_t(1)=\beta$, $t$-almost  {everywhere}  in $(0,T]$.
	\end{enumerate}
\end{definition}

\begin{rem}
Observe that item (3) of Definition \ref{reacdif} and item (2) of Definition \ref{dif} are needed in order to prove uniqueness of weak solutions, see Appendix \ref{appendix}. We prove that the solution satisfies  those two  items  in Section \ref{boundaries} for the case $\hat \kappa =0$. In the case of Definition \ref{reacdif} with $\hat\kappa, \hat c_2 >0$, we prove item (3) as a consequence of  item (1), as  in \cite{BGJO}.
\end{rem}

\begin{definition}
	\label{Rob}
	Let $\hat c_2 > 0$, $\hat m,\hat \kappa \geq 0$ and  let $g:[0,1]\rightarrow [0,1]$ be a measurable function. We say that  $\rho:[0,T]\times[0,1] \to [0,1]$ is a weak solution of the diffusion  equation with inhomogeneous Robin boundary conditions and initial condition $g$:
	\begin{equation}
	\label{eq:dif Rob}
	\begin{cases}
	&\partial_{t} \rho_{t}(u)= \hat c_2 \Delta \rho_t(u),  \quad (t,u) \in [0,T]\times(0,1),\\
	&\partial_u\rho_t(0)=\tfrac{\hat m}{ \hat c_2}(\rho_t(0)-\alpha), \quad \partial_u\rho_t(1)=\tfrac{\hat m}{ \hat c_2}(\beta-\rho_t(1)), \quad t \in [0,T]\\
	&{\rho}_{0}(u)= g(u),\quad u \in (0,1),
	\end{cases}
	\end{equation}
	if  for all $t\in [0,T]$ and {all functions}  {$G \in \mathcal{S}_{Rob}(\hat m,\hat c_2)$}  we have that 
		\begin{equation}
		\label{eq:Rob}
		\begin{split}
		F_{{Rob}}(t, \rho,G,g):=&\left\langle \rho_{t},  G_{t} \right\rangle -\left\langle g,   G_{0}\right\rangle - \int_0^t\left\langle \rho_{s},\Big(\partial_s +  \hat c_2\Delta \Big) G_{s}  \right\rangle ds=0.\\
		\end{split}   
		\end{equation}

\end{definition}

\begin{rem}
    Observe that   \eqref{eq:dif Rob} with  $\hat m=0$ becomes  the heat equation with Neumann boundary conditions.
\end{rem}

\subsection{Hydrodynamic limit}
Denote by ${\mathcal M}^+$ the space of positive measures on $[0,1]$ with total mass bounded by $1$ equipped with the weak topology. The empirical measure $\pi^{N}(\eta,du) \in \mathcal M^+$ is defined, for any configuration  $\eta \in \Omega_{N}$, by 
\begin{equation}\label{MedEmp}
\pi^{N}(\eta, du):=\dfrac{1}{N-1}\sum _{x\in \Lambda_{N}}\eta(x)\delta_{\frac{x}{N}}\left( du\right),
\end{equation}
where $\delta_{a}$ is a Dirac mass on $a \in [0,1]$. We use the notation
$\pi^{N}_{t}(du):=\pi^{N}(\eta_t^N, du).$

Fix $T>0$. We denote by $\mathbb P _{\mu _{N}}$ the probability measure in the Skorohod space $\mathcal D([0,T], \Omega_N)$ induced by the  Markov process $\{\eta_{t}^N\}_{t\in[0,T]} $ with initial distribution $\mu_N$. Moreover, we denote by $\mathbb E_{\mu _{N}}$ the expectation with respect to $\mathbb {P}_{\mu _{N}}$. Let $\lbrace{ Q}_{N}\rbrace_{N>1}$ be the  sequence of probability measures on $\mathcal D([0,T],\mathcal{M}^{+})$ induced by the  Markov process $\lbrace \pi_{t}^{N}\rbrace_{t\in [0,T]}$ and by $\mathbb{P}_{\mu_{N}}$.

\begin{definition}
	Let $\rho_0: [0,1]\rightarrow[0,1]$ be a measurable function. We say that a sequence of probability measures $\lbrace\mu_{N}\rbrace_{N > 1 }$ on $\Omega_{N}$  is associated with the profile $\rho_{0}(\cdot)$ if for any continuous function $G:[0,1]\rightarrow \mathbb{R}$  and every $\delta > 0$ 
	\begin{equation*}
	\lim _{N\to\infty } \mu _{N}\left( \eta \in \Omega_{N} : \bigg| \dfrac{1}{N-1}\sum_{x \in \Lambda_{N} }G\left(\tfrac{x}{N} \right)\eta(x) - \langle G,\rho_{0}\rangle \bigg|    > \delta \right)= 0.
	\end{equation*}
\end{definition}
The next statement is the main theorem of this work.

\begin{thm}[Hydrodynamic limit]
	
	\label{theo:hydro_limit}
	\quad
	
	Let $g:[0,1]\rightarrow[0,1]$ be a measurable function and let $\lbrace\mu _{N}\rbrace_{N> 1}$ be a sequence of probability measures in $\Omega_{N}$ associated with $g(\cdot)$. Then, for any $0\leq t \leq T$,
	\begin{equation*}\label{limHidreform}
	\lim _{N\to\infty } \PP_{\mu _{N}}\left( \eta_{\cdot}^{N} \in \mathcal D([0,T], {\Omega_{N}}) : \bigg| \dfrac{1}{N-1}\sum_{x \in \Lambda_{N} }G\left(\tfrac{x}{N} \right)\eta_{t}^N(x) - \langle G,\rho_{t}\rangle \bigg|    > \delta \right)= 0,
	\end{equation*}
	where the time scale $\Theta(N)$ is given by 
	\begin{equation}
	\Theta(N)=\begin{cases} N^{2 + \theta} & \theta<0;\\
	\frac{N^{2}}{\log(N)} & \theta\geq 0;
	\end{cases}
	\label{timescale}
	\end{equation}
	and  $\rho$ is the unique weak solution of:
	\begin{itemize}
		\item [$\bullet$] \eqref{eq:reacdif Equation} with $\hat \kappa=\kappa$ and $\hat c_2=0$ if $\theta <0$;
		\item [$\bullet$] \eqref{eq:reacdif Equation} with $\hat  \kappa= \kappa $ and $\hat c_2=c_2$, if $\theta =0$;
		\item [$\bullet$] \eqref{eq:dif Equation} with $\hat c_2=c_2$,  if $\theta\in (0,1)$;
		\item [$\bullet$] \eqref{eq:dif Rob} with $\hat  m= m>0 $ and $\hat c_2=c_2>0$, if $\theta=1$;
		\item[$\bullet$] \eqref{eq:dif Rob} with $\hat  m=0 $ and $\hat c_2=c_2>0$, if $\theta>1$.
	\end{itemize}
\end{thm}

The strategy of the proof  is by means of the entropy method, first introduced in \cite{GPV}, and it is divided into two  fundamental steps.  The first one consists in  showing that the sequence $\{Q_N\}_{N>1}$ admits limit points (Section \ref{2}). The second consists in characterizing uniquely the limit point, by showing that it is is a delta measure supported on the trajectories of measures that are absolutely continuous with respect to the Lebesgue measure and whose density is the unique weak solution of the hydrodynamic equation. This program is achieved in Section \ref{3}, where we prove that  the density satisfies the integral formulation of the respective  hydrodynamic equation. Finally, in Appendix \ref{appendix} we prove the uniqueness of  weak solutions of the hydrodynamic equations.

\section{Characterization of limit points}\label{3}
We start assuming that the sequence of measures {$\{Q_N\}_{N>1}$ } has a subsequence converging weakly  to some measure $Q$.  In fact, this is true and it  follows from Proposition \ref{TIGHT}.  Moreover, since we work with an exclusion process, the limiting measure $Q$ is  concentrated on  a trajectory of measures that are absolutely continuous with respect to the Lebesgue measure, that is,  of the form $\pi_{\cdot}(du)=\rho_{\cdot}(u)du$. The proof of this result  is quite standard  and  we do not repeat it here, but we refer the interested reader to, for example,  Section 2.10 of \cite{patricianote}. Therefore,  we want to show that the limit point $Q$ is concentrated on measures whose  densities $\rho$ are unique weak solutions of the respective  hydrodynamic equations. This is the content of the next proposition.

\begin{prop}\label{cara}
	For any limit point $Q$ of the sequence $\{Q_N\}_{N>1}$ it holds
	\begin{itemize}
		\item for $\theta \leq 0$:
			\begin{equation*}\small
		Q\left(\pi_{\cdot} : 	{F_{{Reac}}}(t, \rho,G,g)=0, \forall t \in [0,T], \forall G \in C_c^{1,2}([0,T]\times [0,1]) \right)=1.\normalsize
		\end{equation*}
		{where $\hat c_2=c_2$ for $\theta=0$ and $\hat c_2=0$ for $\theta<0$.}
		\item for $\theta \in (0,1)$:
		\begin{equation*}\small
		Q\left(\pi_{\cdot}: 	F_{{Dir}}(t, \rho,G,g)=0, \forall t \in [0,T], \forall G \in \mathcal S_{Dir}\right)=1.
		\end{equation*}\normalsize
	\item 	for {$\theta \geq 1$:}
		\begin{equation*}\small
		Q\left(\pi_{\cdot} : {	F_{{Rob}}(t, \rho,G,g)=0}, \forall t \in [0,T],\forall G \in  {\mathcal{S}_{Rob}(\hat m,\hat c_2)} \right)=1.\normalsize
		\end{equation*}
		{where for $\theta=1$, $\hat m=m$ and  {$\hat c_2=c_2$}; and for $\theta>1$, $\hat m=0$.}
	\end{itemize}

\end{prop}
The proof of last proposition can be easily adapted from the one in 
	 \cite{BGJO} {(see Proposition 7.1 there)}  by using the results presented in the next sections.

\subsection{Heuristics for hydrodynamic equations}\label{1}
 The aim of this section is show that  density $\rho$  satisfies the integral equations given in the definition of weak solutions. To that end, 
the starting point is  Dynkyn's formula (see, for example, Lemma 5.1 of \cite{KL}) from which, for any $t \in [0,T]$ and any test function {$G \in C^{1,2}([0,T]\times[0,1])$},
\begin{equation}
\label{dynkyn}
M^N_t(G_t)=\langle \pi_t^N, G_t\rangle - \langle \pi_0^N, G_0\rangle - \int_0^t (\Theta(N)L_N+\partial_s)\langle \pi_s^N, G_s\rangle ds
\end{equation}
is a martingale with respect to the natural filtration $\{\mathcal{F}_t=\sigma(\eta_s : s\leq t)\}_{t \in [0,T]}$. Hereinafter, we avoid the dependence on time from the test function $G$, since it does not require any particular analysis and it would only make the notation heavier.  Above the notation $\langle \pi_t^N, G\rangle$ denotes the integral of the function $G$ with respect to the measure $\pi^N_t(du)$ given in \eqref{MedEmp}.
The idea is to prove that, for any $t \in [0,T]$, we can write the martingale in \eqref{dynkyn} as $$M^N_t(G)=\langle \pi_t^N, G\rangle - \langle \pi_0^N, G\rangle - \int_0^t \langle \pi_s^N, \mathcal{A}G\rangle ds,$$ plus some term which vanishes in $L^1$, as $N \rightarrow \infty$ and with a certain   operator $\mathcal{A}$ acting on the test function $G$.
In this way, heuristically speaking, 
	Under the hypothesis of convergence of the empirical measures $\{\pi^N_t\}_{N>1}$ to absolutely continuous measures with respect to the Lebesgue measure on $[0,1]$ with a density $\rho_{t}$,  by proving that the martingale $M_t^N(G)$ vanishes in $L^2$, then  we would obtain that the limit as $N\rightarrow \infty$ in $L^1$ of the last display would be equal to $$0=\langle \rho_t,G\rangle-\langle \rho_o,G\rangle-\int_0^t\int_0^1 (\mathcal{A}G)(u)\rho_s(u)duds.$$ This will correspond to the integral formulation of the weak solution of the hydrodynamic equations that we introduced in Section \ref{HE}.

In order to do this,  we first study the integral term in \eqref{dynkyn},
 which is the one relying on the specific form of the dynamics. Simple, but long  computations (see also  (3.2) of \cite{BGJO}), give  
\begin{equation}
\label{L}
\begin{split}
\Theta(N) L_N \langle \pi_s^N, G \rangle  =& \cfrac{\Theta(N)}{N-1} \sum_{x\in \Lambda_N}  (\mathcal L_NG)(\tfrac{x}{N}) \eta^N_s(x) \\
+& \cfrac{ \kappa_N(\theta) \Theta(N)}{(N-1)N^{\theta}} \sum_{x \in \Lambda_N}  G(\tfrac{x}{N}) \left( r_{N}^{-}(\tfrac{x}{N}) (\alpha- \eta^N_s(x))+   r_{N}^{+}(\tfrac{x}{N}){(\beta- \eta^N_s(x) )}\right),
\end{split}
\end{equation}
where 
\begin{equation}\label{eq:operador_LN}
\mathcal{L}_NG(\tfrac{x}{N}):=\sum_{y \in \Lambda_N}\big(G(\tfrac{y}{N})-G(\tfrac{x}{N})\big)p(y-x)
\end{equation}
and for any $x \in \Lambda_N$
\begin{equation}\label{eq:function_r_pm}
	r^-_N(\tfrac{x}{N}):=\sum_{y\geq x}p(y) \quad \text{ and } \quad 	r^+_N(\tfrac{x}{N}):=\sum_{y\leq N-x}p(y).
\end{equation}
Now, we analyze each regime of  $\theta \in \RR$ separately by taking the corresponding space of test functions and the respective time scale of \eqref{timescale}. The reader can combine these computations  and follow the proof of Proposition 7.1 of \cite{BGJO} in order to conclude.
\subsubsection{Case $\theta<0$}
Recall from \eqref{eq:reac} and  \eqref{timescale}  that in this regime we consider test functions $G \in C_c^{\infty}([0,1])$ and  $\Theta(N)=N^{2+\theta}$. The first term on the right-hand side of \eqref{L} can be studied by using the fact that $G^{(d)}(0)=G^{(d)}(1)=0$ for any $d\geq0$ and  also that both $G,G',\Delta G$  are uniformly bounded.  Indeed, \eqref{L} can be written as
\begin{equation}
\begin{split}
\frac{N^{2+\theta}}{N-1} \sum_{x \in \Lambda_N}\sum_{y \in \Lambda_N} \big(G(\tfrac{y}{N})-G(\tfrac{x}{N})\big)&p(y-x)\eta_s^N(x),
\end{split}
\end{equation}
and since  for any $x\in\Lambda_N $ we have  $\eta_t^N(x)\leq 1$
together with  a Taylor expansion on $G$, last display is bounded from above by a constant times
\begin{equation}
\begin{split}
 {\Big|}\frac{N^{1+\theta}}{N-1} \sum_{x \in \Lambda_N}G'(\tfrac{x}{N})\sum_{y=1-x}^{N-1-x} y^{-2} {\Big|}\lesssim N^{\theta} \sum_{x \in \Lambda_N} {\Big|}G'(\tfrac{x}{N}) {\Big|}x^{-1}\lesssim \frac{N^{\theta}}{\log(N)},
\end{split}
\end{equation}
plus a term that vanishes, as $N\to\infty$.
Observe that last display also  vanishes, as $N\to\infty$, since  $\theta<0$.
The last two terms on the right-hand side of \eqref{L}  are analyzed in the same way as  in  Section 3.1 of \cite{BGJO}, by using Lemma 3.3 of \cite{BJ}, the fact that $G \in C_c^{\infty}([0,1])$ and the hypothesis on the convergence of $\{\pi_{\cdot}^N\}_{N>1}$. We omit these computations here and leave the details to the reader.

\subsubsection{Case $\theta =0$}
In this regime  $G \in C_c^{\infty}([0,1])$ and  $\Theta(N)=N^{2+\theta}$.  The first term on the right-hand side of \eqref{L} can be treated using Lemma \ref{L1conv}. Thus,  we can rewrite it as $c_2\langle \Delta G, \pi_s^N\rangle $ plus terms vanishing, as $N\to\infty$. Indeed, since $G'(0)=G'(1)=0$, by looking at the statement of Lemma \ref{L1conv},  the term on the right-hand side of \eqref{stat} can be rewritten as
{\begin{equation*}
\cfrac{1}{\log(N)} \sum_{x,y \in \Lambda_N}\eta_s^N(x)G'(\tfrac{x}{N})(y-x)p(y-x).
\end{equation*}
From a Taylor expansion on $G$ (recall that $G'(0)=0$ since $G$ has compact support),  last expression  can be bounded from above by a constant times
\begin{equation*}
\begin{split}
  {\Big|}\cfrac{1}{N\log(N)} \sum_{x \in \Lambda_N}x\sum_{y =1+x}^{N-1+x}y^{-2} {\Big|}\lesssim \cfrac{1}{N\log(N)} \sum_{x \in \Lambda_N}1\lesssim \frac{1}{\log(N)},
\end{split}
\end{equation*} plus a term that vanishes, as $N\to\infty$. 
By taking $N\to\infty$, last display also vanishes.} To treat the second term on the right-hand side of \eqref{L}, we have to be a bit careful. Recall that  $\kappa_N(0)=\kappa \log(N)$.
 Therefore, since $G \in C_c^{\infty}([0,1])$, from  Lemma 3.3 of \cite{BJ}, the second term on the right-hand side of \eqref{L} can be rewritten as some term vanishing, as $N\to\infty$, plus
$
\kappa \langle G,V_0\rangle -\kappa \langle G V_1, \pi_s^N\rangle,
$
where $V_0$ and $V_1$ were defined in \eqref{eq:V0V1}.

\subsubsection{Case $\theta \in (0,1)$}
Recall from \eqref{eq:dif} that in this case we consider $G \in \mathcal{S}_{Dir}$  (and recall that we are ignoring the time dependence on $G$). Observe from \eqref{timescale} that $\Theta(N)=N^2/\log(N)$.
The first term on the right-hand side of $\eqref{L}$ is treated exactly as in the case $\theta=0$.
Let us now analyze  the second term on the right-hand side of \eqref{L} focusing on the term involving $r^-_N$, since the other one can be studied in a similar way. Following the same computations as in Section 3 of \cite{mio}, it is possible to show that $r^-_N$ is bounded by a term of order $x^{-2}$ and, since $\eta_s^N$ is uniformly bounded, we can bound from above the whole term by a constant times
\begin{equation}
\begin{split}
 {\Big|}\cfrac{ \kappa N}{N^{\theta}\log(N) } \sum_{x \in \Lambda_N}  G(\tfrac{x}{N})\frac{1}{x^2} {\Big|}=\cfrac{ \kappa}{N^{\theta}\log(N) }  {|}G'(0) {|}\sum_{x \in \Lambda_N} \frac{1}{x}+ \cfrac{ \kappa}{N^{\theta+1}\log(N) } \sum_{x \in \Lambda_N}  {|}G''(\xi) {|}
\end{split}
\end{equation}
for some $\xi \in (0,\tfrac{x}{N})$. Note that the first term on the right-hand side of the previous display is of order $N^{-\theta}$ and the second one is of order $(N^{\theta}\log(N))^{-1}$, hence the whole term vanishes, as $N\to\infty$.

\subsubsection{Case $\theta =1$}
Recall from \eqref{eq:Rob} that in this case we take test functions  {$G \in \mathcal{S}_{Rob}(m, c_2)$}, but we assume that they do not depend on time and recall from \eqref{timescale} that $\Theta(N)=N^2/\log(N)$.
The first term on the right-hand side of \eqref{L} can be analyzed using Lemma \ref{L1conv}. and it is equal to 
\begin{equation}\label{thet1_1}
\cfrac{c_2}{N-1}\sum_{x \in \Lambda_N}\eta_s^N(x)\Delta G(\tfrac{x}{N})+\cfrac{1}{\log(N)} \sum_{x,y \in \Lambda_N}G'(\tfrac{x}{N})(y-x)p(y-x)\eta_s^N(x),
\end{equation}
plus a term that vanishes, as $N\to\infty$. Let us focus on the term on the right-hand side of the previous display. By extending the sum in $y$ to the whole set $\mathbb{Z}$ and by the symmetry of  $p(\cdot)$, we can rewrite that term  as
\begin{equation}
\label{quirep}
\cfrac{1}{\log(N)} \sum_{x\in \Lambda_N}G'(\tfrac{x}{N})\eta_s^N(x)\big(\Theta_x^--\Theta_x^+\big)
\end{equation}
where 
\begin{equation}
\label{theta+-}
\Theta_x^-:=\sum_{y\leq 0}(x-y)p(x-y)\quad \text{ and } \quad \Theta_x^+:= \sum_{y\geq N}(y-x)p(y-x).
\end{equation}

{Let us now  introduce the following quantities. For  $s \in [0,T]$, $N>1$ and $\epsilon>0$, we define
\begin{equation}
\label{media}
\overrightarrow{\eta}^{\epsilon N}_s(0):=\frac{1}{\epsilon N}\sum_{x=1}^{\epsilon N}\eta^N_s(x) \quad \text{ and } \quad \overleftarrow{\eta}^{\epsilon N}_s(N):=\frac{1}{\epsilon N}\sum_{x=N-\epsilon N}^{N-1}\eta^N_s(x).
\end{equation}
Above  $\epsilon N$ should be understood as  $\lfloor \epsilon N\rfloor$ and   $\overrightarrow{\eta}^{\epsilon N}_{s}(0)=\langle \pi^N_{s},\iota_\epsilon^0\rangle$ and  $\overleftarrow{\eta}^{\epsilon N}_s(N)=\langle \pi^N_{s},\iota_\epsilon^1\rangle.$
where 
$\iota_\epsilon^0(u)=\tfrac{1}{\epsilon}{\bf{1}}_{(0,\epsilon)}(u)$ and $\iota_\epsilon^1(u)=\tfrac{1}{\epsilon}{\bf{1}}_{(1-\epsilon,1)}(u)$. The same remark holds also for the definition of the left average.  Heuristically,  since the limiting trajectory consists of absolutely continuous  measures with respect to the Lebesgue measure, we have that  for $r\in\{0,1\}$, $\langle \pi^N_{s},\iota_\epsilon^r\rangle$  converges, when $N\to\infty$, to $$\langle\pi_{s},\iota_\epsilon^r\rangle=\int_0^1 \rho_s(u)\iota_\epsilon^r(u)\,du$$ where  $\rho$ is the density profile that we aim to characterize. Then, by taking the limit as $\epsilon\to 0$ we obtain that $\langle\pi_{s},\iota_\epsilon^r\rangle$ converges to $\rho_s(r)$, as $\epsilon\to 0$,  for all $s\in[0,T]$.
Putting together these observations we get that
$\lim_{\epsilon\to 0}\lim_{N \rightarrow \infty} \langle \pi^N_{r},\iota_\epsilon^r\rangle=\rho_s(r) $ for $r\in\{0,1\}$ in  the $ L^1$ sense.}

Now, from Lemma \ref{replacement1}, we can rewrite \eqref{quirep} as 
\begin{equation}\label{restoLN}
\cfrac{1}{\log(N)} \overrightarrow \eta_s^{\epsilon N}(0)\sum_{x\in \Lambda_N}G'(\tfrac{x}{N})\Theta_x^--\cfrac{1}{\log(N)} \overleftarrow \eta_s^{\epsilon N}(N)\sum_{x\in \Lambda_N}G'(\tfrac{x}{N})\Theta_x^+
\end{equation}
plus some terms vanishing in $L^1$, as $N\to\infty$. Now, we are going to analyze in detail just the term on the left-hand side of last display, since the other one can be analyzed  analogously. By a Taylor expansion on $G'$ around $0$,  we can rewrite that term as
\begin{equation*}
\cfrac{1}{\log(N)}G'(0) \overrightarrow \eta_s^{\epsilon N}(0)\sum_{x\in \Lambda_N}\Theta_x^-+\cfrac{1}{N\log(N)} \overrightarrow \eta_s^{\epsilon N}(0)\sum_{x\in \Lambda_N}G''(\xi)\Theta_x^-
\end{equation*}
for some $\xi \in (0, \tfrac{x}{N})$. From Lemma \ref{thetaconv}, the leftmost term of last display can be rewritten as $c_2G'(0)\overrightarrow \eta_s^{\epsilon N}(0)$ plus a term vanishing, as $N\to\infty$. Clearly, since $G''$ is uniformly bounded, reasoning in the same way we can conclude that the rightmost term of the previous display vanishes, as $N\to\infty$. Analogously, it is possible to show that the rightmost term of \eqref{restoLN} can be written as $-c_2G'(1)\overleftarrow \eta_s^{\epsilon N}(N)$ plus terms vanishing in $L^1$, as $N\to\infty$.

Finally, we analyze the rightmost term on the right-hand side of \eqref{L}. We will focus on the part involving $r^-_N$, since the other one can be analyzed in an analogous way. Recall that in this regime $\kappa_N(1)=\kappa \log(N)$.
Then the term we need to study is equal to 
\begin{equation*}
\cfrac{\kappa N}{N-1} \sum_{x \in \Lambda_N}  G(\tfrac{x}{N}) r_{N}^{-}(\tfrac{x}{N}) (\alpha- \eta^N_s(x)).
\end{equation*}
{From Remark \ref{RMK1},  the time  integral from $0$ to $t$ of this term can be replaced  by the same time integral of }
\begin{equation*}
\cfrac{\kappa N}{N-1} \sum_{x \in \Lambda_N}  G(\tfrac{x}{N}) r_{N}^{-}(\tfrac{x}{N}) (\alpha-\overrightarrow \eta_s^{\epsilon N}(0) ),
\end{equation*}
plus a term vanishing in $L^1$, as $N\to\infty$. Then, by a Taylor expansion on $G$ around $0$ we can rewrite last display as
\begin{equation}
(\alpha-\overrightarrow \eta_s^{\epsilon N}(0) )\kappa \sum_{x \in \Lambda_N} r_{N}^{-}(\tfrac{x}{N})\Big[G(0)+\tfrac{x}{N}  G'(\xi)
\Big]
\end{equation}
for some $\xi \in (0,\tfrac{x}{N})$. {Observe that, since $r^-_N$ is bounded from above by a term of order $x^{-2}$, the second term of the previous display can be bounded from above by:
\begin{equation}
 {\Big|}\cfrac{\kappa(\alpha-\overrightarrow \eta_s^{\epsilon N}(0) ) }{N} \sum_{x \in \Lambda_N} xr_{N}^{-}(\tfrac{x}{N})G'(\xi) {\Big|}\lesssim \cfrac{\kappa }{N} \sum_{x \in \Lambda_N} xr_{N}^{-}(\tfrac{x}{N})\lesssim \frac{\log(N)}{N},
\end{equation}
which vanishes, as $N\to\infty$.}
 
Finally, thanks to the fact that $\sum_{x \in \Lambda_N}r^-_N(\tfrac{x}{N})$ converges to $m$, as $N\to\infty$ (for details see equation (3.7) of \cite{BGJO}), the whole term that we are analyzing can be replaced by
$
\kappa mG(0)(\alpha -\overrightarrow \eta^{\epsilon N}_s(0))
$, plus some error which vanishes, as $N \rightarrow \infty$.
 In a similar way it is not difficult to show that the term involving $r^+_N$ in last term of \eqref{L} can be rewritten as $
\kappa mG(1)(\beta -\overleftarrow \eta^{\epsilon N}_s(1))
$, plus terms vanishing in $L^1$, as $N\to\infty.$ Putting all this together, we conclude that the last integral in \eqref{dynkyn} can be rewritten, under the hypothesis of convergence of the sequence of empirical measures and the properties of functions in  { $\mathcal{S}_{Rob}(m,c_2)$} , as
\begin{equation*}
\int_0^t\int_0^1c_2\Delta G(u)\rho_s(u) duds ,
\end{equation*}
plus terms vanishing in $L^1$, as $N\to\infty$.

\subsubsection{Case $\theta>1$}
As in the previous case, recall from \eqref{eq:Rob} that the test functions  {$G \in \mathcal{S}_{Neu}$}  and we assume that they do not depend on time, and recall from \eqref{timescale} that $\Theta(N)=N^2/\log(N)$. 
The analysis of the first term in \eqref{L} is the same as in the case $\theta=1$. It remains  to study  the last term on the right-hand side of \eqref{L}. We focus on the part involving $r^-_N$, since the other one can be studied in an analogous way. Recall that in this regime $\kappa_N(\theta)=\kappa$.
By using the fact that 
the sum $\sum_{x \in \Lambda_N}r^-_N(\tfrac{x}{N})$ is convergent, $G$ and $(\alpha-\eta_s^N(x))$ are uniformly bounded, we have that
\begin{equation*}
 {\Big|}\cfrac{ \kappa N^{1-\theta}}{(N-1)\log(N)} \sum_{x \in \Lambda_N}  G(\tfrac{x}{N}) r_{N}^{-}(\tfrac{x}{N}) (\alpha- \eta^N_s(x)) {\Big|}\lesssim \frac{N^{1-\theta}}{\log(N)}
\end{equation*}
which  vanishes, as $N\to\infty$.

 {Now, we introduce the ingredients in order to prove properly some of the results above. }

\subsection{Dirichlet form and relative entropy}
 For a probability measure $\mu$ on $\Omega_N$ and a density function $f$ with respect to $\mu$, we introduce  the quadratic form  given  by $D_N:=D^0_N+D^l_N+D^r_N$ where 
\begin{eqnarray} \label{D}
&D_{N}^{0}(\sqrt{f},\mu):=\tfrac{1}{2}\int\sum_{x,y\in\Lambda_N}p(y-x)\left(\sqrt{f(\sigma^{x,y}\eta)}- \sqrt{f(\eta)}\right)^{2} d\mu, \\
&D_{N}^{l}(\sqrt{f},\mu):=\tfrac{\kappa_N(\theta)}{N^{\theta}}\int\sum_{x\in\Lambda_N}r_N^-(\tfrac{x}{N})c_{x}(\eta;\alpha)\left( \sqrt {f{(\sigma^{x}\eta)}}-\sqrt {f(\eta)}\right)^{2}d\mu,\\&
D_{N}^{r}(\sqrt{f},\mu):=\tfrac{\kappa_N(\theta)}{N^{\theta}}\int\sum_{x\in\Lambda_N}r_N^+(\tfrac{x}{N})c_{x}(\eta;\beta)\left( \sqrt {f{(\sigma^{x}\eta)}}-\sqrt {f(\eta)}\right)^{2}d\mu.
\end{eqnarray}
In equation (5.4) of \cite{BGJO} it was obtained  a relationship between $D_N(\sqrt{f},\mu)$ and the Dirichlet form $\langle L_N\sqrt f, \sqrt f\rangle_{\mu}$  {for a certain choice of $\mu$. We recall it here.} 
 {Let $\nu_{h(\cdot)}$ be} the Bernoulli product measure  defined  by its marginals
		 \begin{equation}\label{prod}
		 \nu_{h(\cdot)}(\eta \in \Omega_N : \eta(x)=1)=h(\tfrac{x}{N})\end{equation} with  $h: [0,1]\rightarrow [0,1]$ a  Lipschitz  function  satisfying $h(0)=\alpha$ and $h(1)=\beta$. Then  
\begin{equation}
\label{estimateDirLip}
\langle L_N\sqrt f, \sqrt f\rangle_{ {\nu_{h(\cdot)}}}\lesssim - \frac{1}{4}D_N(\sqrt f, \nu_{h(\cdot)})+ \frac{\log(N)}{N} + \frac{1}{N^{\theta}}.
\end{equation}
If $h$ is  a constant function then
\begin{equation}
\label{estimateDir}
\langle L_N\sqrt f, \sqrt f\rangle_{ {\nu_{h(\cdot)}}}\lesssim - \frac{1}{4}D_N(\sqrt f, \nu_{h(\cdot)}) + \frac{1}{N^{\theta}}.
\end{equation}

	We state now a relative entropy estimate which is fundamental in the proof of the next results. For two measures $\mu$ and $\tilde \mu$ both defined in $\Omega_N$, the relative entropy between    $\mu$ and $\tilde \mu$ is denoted by $H(\mu|\tilde \mu)$ and it2 is defined by $H(\mu|\tilde \mu):=\sum_{\eta\in \Omega_N}\mu(\eta)\log\bigg(\frac{\mu(\eta)}{\tilde \mu (\eta)}\bigg).$ In the case of an exclusion process evolving on a finite state space, it is easy to show that, for any product measure $\nu_{h(\cdot)}$ and any probability measure $\mu_N$ on $\Omega_N$, the following estimate holds
	\begin{equation}\label{entropy estimate}
	H(\mu_N|\nu_{h(\cdot)}) \, {\leq} \,K_0 N,
	\end{equation}
	where $K_0$ is a positive constant depending only on $\alpha$ and $\beta$. The proof of the previous estimate can be found, for example, in equation (5.1)  of \cite{BGJO}.

\subsubsection{Proof of item (1) of Definition \ref{reacdif}}

The proof of this item  follows Section 6.2 of \cite{BGJO}, so we omit many details.  Observe that
	\begin{equation*}
	\begin{split}
	&\mathbb{E}_{\mu_N}\bigg[\int_0^T N\sum_{x \in \Lambda_N} G(\tfrac{x}{N})r^-_N(\tfrac{x}{N})(\alpha - \eta_s^N(x))ds\bigg]
	\end{split}
	\end{equation*}
	from  {the entropy estimate and Jensen's inequality, last dispaly can be bounded from above  by
	\begin{equation}\label{exp}
	\begin{split}
	\frac{H(\mu_N|\nu_{h(\cdot)})}{N}+\frac{1}{N}\log \mathbb{E}_{\nu_{h(\cdot)}}\bigg[e^{N\big|\int_0^T N\sum_{x \in \Lambda_N} G\Big(\tfrac{x}{N}~\Big)r^-_N\Big(\tfrac{x}{N}\Big)(\alpha - \eta_s^N(x))ds\big|}\bigg].
	\end{split}
	\end{equation}
Now we can use \eqref{entropy estimate}. Moreover, by removing the absolute value in the exponential in last display (which is possible thanks to the fact that $e^{|x|}\leq \max 
\{e^x,e^{-x}\}$ and since
$$\limsup_N \log(a_N+b_N) \leq \max\{\limsup_N \log (a_N), \limsup_N \log(b_N)\}$$
and Feynman-Kac formula  we can bound \eqref{exp} from above by }  $K_0$ plus 
	\begin{equation}\label{lest}
	\begin{split}
	\sup_f \bigg\{\int N\sum_{x \in \Lambda_N} G(\tfrac{x}{N})r^-_N(\tfrac{x}{N})(\alpha - \eta(x))f(\eta) d \nu_{h(\cdot)}+\tfrac{\Theta(N)}{N}\langle L_N \sqrt f, \sqrt f\rangle_{\nu_{h(\cdot)}} \bigg\}.
	\end{split}
	\end{equation}
	Above the supremum {is} carried over all the density functions with respect to $\nu_{h(\cdot)}$. 
	As in the previous proof,  we can rewrite the first term inside the supremum above as
	\begin{equation}\label{boh}
	\begin{split}
	&\int \frac{N}{2}\sum_{x \in \Lambda_N} G(\tfrac{x}{N})r^-_N(\tfrac{x}{N})(\alpha - \eta(x))(f(\eta)-f(\sigma^{x}\eta))d\nu_{h(\cdot)}\\ +&\int \frac{N}{2}\sum_{x \in \Lambda_N} G(\tfrac{x}{N})r^-_N(\tfrac{x}{N})(\alpha - \eta(x))f(\sigma^{x}\eta)\Big(1-\tfrac{\nu_{h(\cdot)}(\sigma^{x}\eta)}{\nu_{h(\cdot)}(\eta)} \Big)d\nu_{h(\cdot)}.
	\end{split}
	\end{equation}
By using the inequality $ab \leq \frac{Aa^2}{2}+\frac{b^2}{2A}$ on the first term of last display,  with $$A=\frac{N^{2+\theta}}{\Theta(N)c_x(\eta;\alpha)\kappa_N(\theta)},$$ we  bound it from above by a constant times
	\begin{equation}\label{boh2}
	\begin{split}
     &\int\frac{N^{3+\theta}}{4\Theta(N)\kappa_N(\theta)}\sum_{x \in \Lambda_N} G(\tfrac{x}{N})^2r^-_N(\tfrac{x}{N})(\alpha - \eta(x))^2(\sqrt f(\eta)+\sqrt f(\sigma^{x}\eta))^2d\nu_{h(\cdot)}\\+& \int \frac{\Theta(N)\kappa_N(\theta)}{4N^{1+\theta}}\sum_{x \in \Lambda_N}r^-_N(\tfrac{x}{N})c_x(\eta;\alpha)(\sqrt f(\eta)-\sqrt f(\sigma^{x}\eta))^2d\nu_{h(\cdot)}.
     \end{split}
	\end{equation}
    Recall \eqref{D} and observe that the last term above can be rewritten as
    \begin{equation}
    \frac{\Theta(N)}{4N}D_N^l(\sqrt{f}, \nu_{h(\cdot)})\leq  \frac{\Theta(N)}{4N}D_N(\sqrt{f}, \nu_{h(\cdot)}).
    \end{equation}
    So, this term summed with the last one inside the supremum in \eqref{lest}, thanks to \eqref{estimateDirLip}, is equal to some term that vanishes, as $N\rightarrow \infty$. {Observe now that there exists} a constant $C>0$ such that the first term of \eqref{boh2} is equal to
    \begin{equation}
    \begin{split}
    &\int_{\Omega_N}\frac{N}{4\kappa}\sum_{x \in \Lambda_N} G(\tfrac{x}{N})^2r^-_N(\tfrac{x}{N})(\alpha - \eta(x))^2(\sqrt f(\eta)+\sqrt f(\sigma^{x}\eta))^2d\nu_{h(\cdot)}\\&\lesssim \frac{1}{4N\kappa}\sum_{x \in \Lambda_N} N^2G(\tfrac{x}{N})^2r^-_N(\tfrac{x}{N})  \int_{\Omega_N}(\sqrt f(\eta)+\sqrt f(\sigma^{x}\eta))^2d\nu_{h(\cdot)}\\ &\lesssim \frac{C}{N}\sum_{x \in \Lambda_N} N^2G(\tfrac{x}{N})^2r^-_N(\tfrac{x}{N}), 
    \end{split}
    \end{equation}
    because $(\alpha - \eta(x))^2$ is uniformly bounded and $f$ is a density with respect to $\nu_{h(\cdot)}$. Then, thanks to Lemma 3.3 of \cite{BJ}, in the limit $N\rightarrow \infty$, we can bound from above the first term of \eqref{boh2} by $C\int_0^1 r^-(u)G(u)^2du$.
	The second term of \eqref{boh} is of order {$O(N^{-1})$} since  it is possible to show  that $\Big(1-\tfrac{\nu_{h(\cdot)}(\sigma^{x}\eta)}{\nu_{h(\cdot)}(\eta)} \Big)$ is of order $N^{-1}$ (see, for example, Section 5 of \cite{BGJO}).
	
	Summarizing, taking the limit $N \rightarrow \infty$ in \eqref{exp} we get
	\begin{equation*}
	\mathbb{E}\bigg[\int_0^T \int_0^1 \frac{G(u)(\alpha- \rho_s(u))}{|u|^2}ds du-\int_0^T \int_0^1 \frac{G(u)^2}{|u|^2} ds du\bigg]\leq C'
	\end{equation*}
	for some positive constant $C'$ which {depends} on $C$. The rest of the proof is completely analogous to the one given in Section 6.2 of \cite{BGJO} so we omit it and we ask the reader to fill in the details.

\subsubsection{Proof of item (3) of Definition \ref{reacdif} and item (2) of Definition \ref{dif}}\label{boundaries}
Observe that for $\theta<1$ we need to show that the weak solutions satisfy Dirichlet boundary conditions.
When $\theta<0$ the solution is explicit so there is nothing to verify, see Remark 2.3 of \cite{BGJO}. In the regime $0\leq \theta\leq 1$ one can follow the arguments of  Section 5.3 in \cite{BGJO} that also fit our case. We leave the details to the reader.

\section{Tightness}\label{2}
Here we show that the sequence of measures $\{Q_N\}_{N>1}$ is tight {and so, in particular, it has limit points.} 
\begin{prop}\label{TIGHT}
	The sequence of measures $\{Q_N\}_{N>1}$ is tight in $\mathcal{D}([0,T],\mathcal{M}^+)$ with respect to the  {Skorokhod} topology.
\end{prop}

\begin{proof} Here we follow the proof of Proposition 4.1 of  \cite{BGJO} and, for that reason, many details are omitted. Observe that, as in  \cite{BGJO}, we prove the result for test functions $G \in C_c^2([0,1])$, and by using an $L^1$ approximation (see \cite{BGJO} for details), we  then extend the result to any $G \in C^0 ([0, 1])$. 
	
		 {The proof follows  by showing the next two items:
\begin{align*}
&(1)\quad \lim_{\delta \rightarrow 0} \limsup_{N \rightarrow \infty} \sup_{\tau \in \mathcal{T}_T, \bar{\tau} \leq \delta} \mathbb{E}_{\mu_N} \left[ \Big| \int_{\tau}^{\tau+ \bar{\tau}} \Theta(N) L_{N}\langle \pi_{s}^{N},G\rangle ds \Big| \right] = 0\\
&(2)\quad\lim_{\delta \rightarrow 0} \limsup_{N \rightarrow \infty} \sup_{\tau \in \mathcal{T}_T, \bar{\tau} \leq \delta} \mathbb{E}_{\mu_N} \left[ \left(  M_{\tau}^{N}(G) -  M_{\tau+\bar{\tau}}^{N}(G) \right)^2 \right] = 0.
\end{align*}
	Above, $\mathcal{T}_T$ is the set of stopping times bounded by $T$ and we assume that all the stopping times are bounded by $T$.}

		 {The item (1) is a consequence of the fact } that  there exists a constant $C>0$ such that, for any $\theta \in \mathbb{R}$ and any $N>1$,
	$
    |\Theta(N) L_{N}\langle \pi_{s}^{N},G\rangle|<C.
  $
 	 {To prove it, recall \eqref{L}. Then,}   we have to show that 
    \begin{equation}
    \label{abs3}
    \begin{split}
    &\cfrac{\Theta(N)}{N-1} \sum_{x\in \Lambda_N}  \big|(\mathcal L_NG)(\tfrac{x}{N})\big|
    + \cfrac{ \kappa_N(\theta) \Theta(N)}{(N-1)N^{\theta}} \sum_{x \in \Lambda_N}  \big|G(\tfrac{x}{N})\big| \left| r_{N}^{-}(\tfrac{x}{N})+   r_{N}^{+}(\tfrac{x}{N})\right|<C.
    \end{split}
    \end{equation}
    Let us start bounding the {first term on the left-hand side} of last display. {By a Taylor expansion on $G$, we can bound that term from above by a constant times
    \begin{equation*}
    \begin{split}
    	\cfrac{\Theta(N)}{N^2} \sum_{x\in \Lambda_N}|G'(\xi)|\sum_{y \in \Lambda_N}|y-x|p(y-x)
    \end{split}
    \end{equation*}
    for some $\xi \in (0,\tfrac{x}{N})$. Last display can be bounded from above by a constant times
    \begin{equation*}
    \begin{split}
    \cfrac{\Theta(N)}{N^2} &\sum_{x\in \Lambda_N}|G'(\tfrac{x}{N})|\sum_{y =x+1}^{N-1+x}yp(y)\lesssim \frac{\log(N)\Theta(N)}{N^2}.
    \end{split}
    \end{equation*}
    So,  for any value of $\Theta(N)$ given in \eqref{timescale} the first term on the left-hand side of \eqref{abs3} is of order $O(1)$.
    Let us now analyze the term on the right-hand side of \eqref{abs3}. First observe that, since $G$ has compact support in $(0,1)$ we have that $r^{-}_N(\tfrac{x}{N})\lesssim x^{-2}$ and $r^+_N(\tfrac{x}{N})\lesssim (N-x)^{-2}$ (see Appendix B of \cite{BJ} for details on this estimate). Hence, the whole term can be bounded from above  {by} \small
    \begin{equation}\label{rest}
    \begin{split}
    &\cfrac{ \kappa_N(\theta) \Theta(N)}{(N-1)N^{\theta}} \sum_{x \in \Lambda_N}  \big|G(\tfrac{x}{N})\big|\big[x^{-2}+(N-x)^{-2}\big]\lesssim\cfrac{ \kappa_N(\theta) \Theta(N)}{N^{\theta+1}} \sum_{x \in \Lambda_N}  x^{-2}\big[|G(\tfrac{x}{N})\big|+|G(\tfrac{N-x}{N})\big|\big].
    \end{split}
    \end{equation} \normalsize
    Now, we use the fact that $G$ has compact support to perform a Taylor expansion of $G(\tfrac{x}{N})$ around $0$ and of $G(\tfrac{N-x}{N})$ around 1, to get that the term on the right-hand side of \eqref{rest} is bounded from above by a constant times
    	$\|G''\|_{\infty}\cfrac{ \kappa \Theta(N)}{N^{\theta+3}} \sum_{x \in \Lambda_N} 1.$ Note that, for any value of $\Theta(N)$ this term vanishes, as $N$ goes to infinity.
So, we showed \eqref{abs3}.

To prove 	 {item (2), we  use the fact that}  $$\displaystyle\left( M^{N}_{t}(G)\right)^{2}-\int^{t}_{0} \Theta(N)\left[ L_{N} \langle\pi^{N}_{s},G \rangle^{2}- 2\langle\pi^{N}_{s},G \rangle L_{N} \langle\pi^{N}_{s},G \rangle\right]ds,$$
   is a martingale with respect to the natural filtration  $\{\mathcal{F}_{t}\}_{t\in [0,T]}$. If we manage to prove that the integrand in the last display is uniformly bounded in $N$, the proof ends. It is possible to bound this integrand} from above by a constant times
    \begin{equation*}
    \label{T6}
    \begin{split}
    &\dfrac{\Theta(N)}{(N-1)^{4}} \sum_{x,y\in\Lambda_{N}} (x-y)^{2}p(x-y)
    +\dfrac{\kappa_N(\theta) \Theta(N)}{(N-1)^{2}N^{\theta}}\sum_{ x\in\Lambda_{N}} \left(G\left(\tfrac{x}{N}\right)\right)^{2} \left( r_{N}^{-}(\tfrac{x}{N}) +r_{N}^{+}(\tfrac{x}{N})  \right).\\
    \end{split}
    \end{equation*}
    The  leftmost term in the previous display is of order 
   {$
   O\Big(\tfrac{\Theta(N)\log(N)}{(N-1)^{3}}\Big),
   $}
   so, for any value of {$\Theta(N)$}, it is of lower order than $O(1)$. To  {estimate the rightmost term in last display, we proceed as in \eqref{rest} to bound it from above by a constant times
   	\begin{equation}
   	\dfrac{\kappa_N(\theta) \Theta(N)}{(N-1)^{2}N^{\theta}}\sum_{ x\in\Lambda_{N}} x^{-2}\big[G(\tfrac{x}{N})^2+G(\tfrac{N-x}{N})^2\big])\lesssim 	\dfrac{\kappa_N(\theta) \Theta(N)}{(N-1)^{2}N^{\theta}} {\lesssim 1},
   	\end{equation}
   	for any value of $\theta$.}
  This concludes the proof.
\end{proof}
\section{Technical lemmas}\label{TEC}
In this section we collect all the results that we used in the previous arguments. 
\subsection{Convergence of discrete operators}
\begin{lem}
	\label{thetaconv}
	Recall \eqref{theta+-}. Then,
	\begin{equation}
	\lim_{N\rightarrow \infty}\frac{1}{\log(N)}\sum_{x\in \Lambda_N}\Theta_x^{\pm}=c_2.
	\end{equation}
\end{lem}
\begin{proof}
	We give in details the proof for the case $\Theta_x^-$, the case with $\Theta_x^+$ is analogous.
	Observe that
	\begin{equation*}
	\frac{1}{\log(N)}\sum_{x\in \Lambda_N}\Theta_x^{-}=\frac{1}{\log(N)}\sum_{x\in \Lambda_N}\sum_{z\geq x}zp(z)=\frac{1}{\log(N)}\bigg[\sum_{z=1}^N\sum_{x=1}^z zp(z)+\sum_{z>N}\sum_{x=1}^{N-1} zp(z)\bigg],
	\end{equation*}
	where the last equality is due to Fubini's theorem. Note that the rightmost term in last display can be estimated by
	\begin{equation*}
	\frac{1}{\log(N)}\sum_{z>N}\sum_{x=1}^{N-1} zp(z)\lesssim \frac{N-1}{\log(N)}\sum_{z>N} z^{-2}\lesssim \frac{1}{\log(N)}
	\end{equation*}
	so, it vanishes, as $N\to\infty$. Now observe that
	\begin{equation*}
	\begin{split}
	\frac{1}{\log(N)}\sum_{z=1}^N\sum_{x=1}^z zp(z)-c_2=\frac{1}{\log(N)}\sum_{z=1}^Nz^2p(z)-c_2=c_2\bigg[\frac{1}{\log(N)}\sum_{z=1}^N\frac{1}{z}-1\bigg],
	\end{split}
	\end{equation*}
	then thanks to the estimate $\sum_{n=1}^{k}n^{-1}\leq \log k +1$, it is easy to conclude.
\end{proof}

\begin{lem}\label{L1conv}
	For any  {$G$ sufficiently smooth} the following limit holds:
	\begin{equation}\label{stat}
	\begin{split}
\lim_{N\rightarrow \infty} \sup_{x \in \Lambda_N}\bigg|\cfrac{N^2}{\log(N)}\mathcal{L}_NG&(\tfrac{x}{N})-c_2\Delta G(\tfrac{x}{N})-G'(\tfrac{x}{N})\cfrac{N}{\log(N)}\sum_{y \in \Lambda_N}(y-x)p(y-x)\bigg|=0,
\end{split}
	\end{equation}
	where $\mathcal{L}_N$ was defined in \eqref{eq:operador_LN}.
\end{lem}

\begin{proof}
By a Taylor expansion on $G$ we can bound from above the term inside the supremum by
	\begin{equation*}
	\begin{split}
	\Big|\cfrac{1}{{2}\log(N)}\Delta G(\tfrac{x}{N})\sum_{y \in \Lambda_N}(y-x)^2p(y-x)-{c_2}\Delta G(\tfrac{x}{N})\Big| 
	+ 	\Big|\cfrac{1}{{6}N\log(N)}\sum_{y \in \Lambda_N}G^{(3)}(\xi)(y-x)^3p(y-x)\Big|,
	\end{split}
	\end{equation*}
	for some $\xi$ between $\tfrac{x}{N}$ and $\tfrac{y}{N}$, {with  $x,y\in\Lambda_N$.} To conclude, it is enough to prove that these two terms vanish, as $N\to\infty$.  Observe that, since $G^{(3)}$ and $\eta_s^N$ are uniformly bounded, the second term of last display can be bounded from above by a constant times
	\begin{equation*}
	\cfrac{1}{N \log(N)}\sum_{y \in \Lambda_N}1 \lesssim \cfrac{1}{\log(N)},
	\end{equation*}
	and  so, it vanishes as $N\to\infty$. 
	The  remaining  term can be bounded from above by 
	\begin{equation*}
	\begin{split}
	&\sup_{x \in \Lambda_N}\Big|\Delta G(\tfrac{x}{N})\Big[\cfrac{1}{2\log(N)}\sum_{y \in \Lambda_N}(y-x)^2p(y-x)-c_2\Big]\Big|\\&\lesssim  \sup_{x \in \Lambda_N}\Big|\cfrac{c_2}{2\log(N)}\sum_{y =1-x}^{N-1-x}y^{-1}\mathbb{1}_{\{y\neq 0\}}-{c_2}\Big|\leq \Big|\cfrac{c_2}{\log(N)}\sum_{y =1}^{N-2}y^{-1}-{c_2}\Big|\\& \lesssim \cfrac{\log(N-2) +1}{\log(N)}-1\lesssim \cfrac{1}{\log(N)},
	\end{split}
	\end{equation*}
	which vanishes, as $N\to\infty$.
	Above, the first inequality follows from $\Delta G$ being uniformly bounded and the second one from the inequality $\sum_{n=1}^{k}n^{-1}\leq \log k +1$. 
	
\end{proof}

\subsection{Replacement lemmas}
The replacement lemmas are technical results which are used to close the equations coming from the Dynkin's formula in terms of the empirical measure.

\begin{lem}\label{replacement1}
	Fix $\epsilon \in (0,1)$. For any $\theta \geq 1$, for any $t \in [0,T]$ and $G \in C^{\infty}([0,1])$, the following limit holds
	\begin{equation}\label{statrep}
	\lim_{N \rightarrow \infty}\mathbb{E}_{\mu_N}\bigg[ {\Big|}\int_0^t\cfrac{1}{\log(N)} \sum_{x\in \Lambda_N}G'(\tfrac{x}{N})\big(\eta_s^N(x)-\overrightarrow{\eta}^{\epsilon N}_s(0)\big)\Theta_x^{-}ds {\Big|}\bigg]=0.
	\end{equation}
	The same result holds for the right boundary, that is,  {by replacing}  $\overrightarrow{\eta}^{\epsilon N}_s(0)$ by $\overleftarrow{\eta}^{\epsilon N}_s(N)$ and $\Theta_x^{-}$ by $\Theta_x^{+}$ (introduced in \eqref{theta+-}).
\end{lem}
\begin{proof}
We prove in detail the limit for the left boundary but the one for the right boundary is completely analogous.  By splitting  the sum,   we  rewrite the integrand function above as 
\begin{equation}\label{statrep2}
\cfrac{1}{\log(N)} \sum_{\substack{x\in \Lambda_N\\x\leq \epsilon N}}G'(\tfrac{x}{N})\big(\eta_s^N(x)-\overrightarrow{\eta}^{\epsilon N}_s(0)\big)\Theta_x^{-}+\cfrac{1}{\log(N)} \sum_{\substack{x\in \Lambda_N\\x > \epsilon N}}G'(\tfrac{x}{N})\big(\eta_s^N(x)-\overrightarrow{\eta}^{\epsilon N}_s(0)\big)\Theta_x^{-}.
\end{equation}
Observe that the term on the right-hand side of the previous display can be bounded by a constant times
\begin{equation*}
\cfrac{1}{\log(N)} \sum_{\substack{x\in \Lambda_N\\x > \epsilon N}}\sum_{y\geq x}c_2y^{-2}\lesssim \cfrac{1}{\log(N)} \sum_{\substack{x\in \Lambda_N\\x > \epsilon N}}x^{-1}\lesssim \frac{1}{\log(N)} \log\Big(\tfrac{N-1}{\epsilon N}\Big)
\end{equation*}
which vanishes, as $N\to\infty.$
To conclude, now we have to make  use of the time integration of the remaining term.   {Therefore the $ L^1(\mathbb P_{\mu_N})$-norm of the time integral from $0$ to $t$ of that term can be bounded from above, by using the entropy inequality and Jensen's inequality by $\tfrac{K_0}{B}$} plus
\begin{equation*}
\frac{1}{BN}\log \mathbb{E}_{\nu_{h}}\Bigg[e^{BN\big|\int_0^t\tfrac{1}{\log(N)} \sum_{\substack{x\in \Lambda_N\\x\leq \epsilon N}}G'(\tfrac{x}{N})\big(\eta_s^N(x)-\overrightarrow{\eta}^{\epsilon N}_s(0)\big)\Theta_x^{-}ds\big|}\Bigg].
\end{equation*}
Above  $K_0$ and $B$ are positive constants,
 {We note that the constant $K_0$ comes from \eqref{entropy estimate}, i.e. it is the price}  to replace the measure $\mu_N$ by $\nu_h$, {the  Bernoulli product measure associated with a constant profile $h(\cdot)\equiv h$.} Moreover, by the Feynman-Kac's formula, last display is bounded from above by a  constant times
\begin{equation}\label{FK1}
\sup_{f}\Bigg\{\int\frac{1}{\log(N)}\sum_{\substack{x\in \Lambda_N\\x\leq \epsilon N}}G'\Big(\tfrac{x}{N}\Big)\big(\eta^N(x)-\overrightarrow{\eta}^{\epsilon N}(0)\big)\Theta_x^{-}f(\eta)d\nu_{h}+\frac{N}{B\log(N)}\langle L_N\sqrt f, \sqrt f \rangle_{\nu_{h}}
\Bigg\},
\end{equation}
where the supremum is carried over all the densities with respect to $\nu_{h}$. Now, we rewrite the sum inside the supremum as follows
\begin{equation*}
\begin{split}
&\sum_{\substack{x\in \Lambda_N\\x\leq \epsilon N}}G'(\tfrac{x}{N})\big(\eta^N(x)-\overrightarrow{\eta}^{\epsilon N}(0)\big)\Theta_x^{-}\\=&\frac{1}{\epsilon N}\sum_{\substack{x\in \Lambda_N\\x\leq \epsilon N}}G'(\tfrac{x}{N})\Theta_x^{-}\sum_{y =1}^{\epsilon N}\big(\eta^N(x)-\eta^N(y)\big)\\=&\frac{1}{\epsilon N}\sum_{\substack{x\in \Lambda_N\\x\leq \epsilon N}}G'(\tfrac{x}{N})\Theta_x^{-}\Big[\sum_{y =1}^{x-1}\sum_{z=y}^{x-1}\big(\eta^N(z+1)-\eta^N(z)\big)-\sum_{y =x+1}^{\epsilon N}\sum_{z=x}^{y-1}\big(\eta^N(z+1)-\eta^N(z)\big)\Big].
\end{split}
\end{equation*}
Let us define 
\begin{equation*}
\begin{split}
	C^{\pm}(\eta,f,x):=&\Big[\sum_{y =1}^{x-1}\sum_{z=y}^{x-1}\big(\eta^N(z+1)-\eta^N(z)\big)\big( f(\eta) \pm f(\sigma^{z,z+1}\eta)\big)\\-&\sum_{y =x+1}^{\epsilon N}\sum_{z=x}^{y-1}\big(\eta^N(z+1)-\eta^N(z)\big( f(\eta) \pm f(\sigma^{z,z+1}\eta)\big)\Big] .
\end{split} 
\end{equation*}
Note now that we can rewrite the first term in the supremum in \eqref{FK1} as
\begin{equation*}
\begin{split}
\int\frac{1}{{\epsilon N}\log(N)}\sum_{\substack{x\in \Lambda_N\\x\leq \epsilon N}}\Theta_x^-G'(\tfrac{x}{N})C^-(\eta,f,x)d\nu_{h}+\int\frac{1}{2{\epsilon N}\log(N)}\sum_{\substack{x\in \Lambda_N\\x\leq \epsilon N}}G'(\tfrac{x}{N})C^+(\eta,f,x)d\nu_{h}.
\end{split}
\end{equation*}
By a change of  variables, the  term  on the {right-hand side} of  the previous display is equal to $0$. From Young's and Cauchy-Schwarz's inequality, we can  bound from above the leftmost term in the previous display  by the sum of
\begin{equation}\label{I}
\begin{split}
\int\frac{1}{\epsilon NA\log(N)}&\sum_{\substack{x\in \Lambda_N\\x\leq \epsilon N}}\Theta_x^-G'(\tfrac{x}{N})^2\Big[\sum_{y =1}^{x-1}\sum_{z=y}^{x-1}\big(\eta^N(z+1)-\eta^N(z)\big)^2\big(\sqrt f(\eta)+\sqrt f(\sigma^{z,z+1}\eta)\big)^2\\&-\sum_{y =x+1}^{\epsilon N}\sum_{z=x}^{y-1}\big(\eta^N(z+1)-\eta^N(z)\big)^2\big(\sqrt f(\eta)+\sqrt f(\sigma^{z,z+1}\eta)\big)^2\Big]d\nu_{h}
\end{split}
\end{equation}
and 
\begin{equation}
\label{II}
\begin{split}
\int\frac{A}{2\epsilon N\log(N)}\sum_{\substack{x\in \Lambda_N\\x\leq \epsilon N}}\Theta_x^-&\bigg[\sum_{y =1}^{x-1}\sum_{z=y}^{x-1}\big(\sqrt f(\eta)-\sqrt f(\sigma^{z,z+1}\eta)\big)^2\\&-\sum_{y =x+1}^{\epsilon N}\sum_{z=x}^{y-1}\big(\sqrt f(\eta)-\sqrt f(\sigma^{z,z+1}\eta)\big)^2\bigg]d\nu_{h},
\end{split}
\end{equation}
for  {any}  positive  constant $A$.
Note that 
\begin{equation*}
\begin{split}
\int\bigg[\sum_{z=y}^{x-1}\big(\sqrt f(\eta)-\sqrt f(\sigma^{z,z+1}\eta)\big)^2-\sum_{z=x}^{y-1}&\big(\sqrt f(\eta)-\sqrt f(\sigma^{z,z+1}\eta)\big)^2\bigg]d\nu_{h}\leq D^N(f,\nu_{h}).
\end{split}
\end{equation*}
Now we choose  $$A=\frac{N}{B\log(N)}\frac{2\log(N)}{\sum_{\substack{x\in \Lambda_N\\x\leq \epsilon N}}\Theta_x^-},$$ which is possible since the  {partial sum in $x$} is convergent if multiplied by $1/\log(N)$, see Lemma \ref{thetaconv}. Then, from \eqref{estimateDir}, we can bound  \eqref{II} by $-\frac{N}{B\log(N)}\langle L_N\sqrt f, \sqrt f \rangle_{\nu_{h}}$ and  this term cancels with the last one inside the supremum in \eqref{FK1}.

It remains to analyze \eqref{I} for  this choice of $A$. Since $\frac{1}{\log(N)}\sum_{x \in \Lambda_N}\Theta_x^-$ is convergent, $G'$ and $\eta$ are bounded, we can bound this term by a constant times
\begin{equation*}
\begin{split}
\frac{B}{\epsilon N^2}\sum_{\substack{x\in \Lambda_N\\x\leq \epsilon N}}\Theta_x^-\int \sum_{y=1}^{x-1}\sum_{z=1}^{x-1}\big(\sqrt f(\eta)+\sqrt f(\sigma^{z,z+1}\eta)\big)^2 d\nu_{h}& \lesssim \frac{B}{\epsilon N^2}\sum_{\substack{x\in \Lambda_N\\x\leq \epsilon N}}x^2\sum_{y \geq x} yp(y)\\& \lesssim \frac{B}{\epsilon N^2}\sum_{\substack{x\in \Lambda_N\\x\leq \epsilon N}}x \lesssim B\epsilon,
\end{split}
\end{equation*}
where the first inequality comes from the fact that $f$ is a density with respect to $\nu_{h}$ and from the inequality $(a+b)^2\lesssim a^2+b^2$.
So, when $\epsilon \rightarrow 0$, this term vanishes.  {Finally, by taking $B$ to infinity, we  conclude} the proof of \eqref{statrep}. 

\end{proof}

\begin{rem}\label{RMK1}
	We observe that following the same steps as in the previous proof, it is also possible to show that
\begin{equation}\label{statrep3}
\lim_{N \rightarrow \infty}\mathbb{E}_{\mu_N}\bigg[ {\Big|}\int_0^t \sum_{x\in \Lambda_N}G(	\tfrac{x}{N})r^{-}_N(\tfrac{x}{N})\big(\eta_s^N(x)-\overrightarrow{\eta}^{\epsilon N}_s(0)\big)ds {\Big|}\bigg]=0
\end{equation}
and the same result for the right boundary, that is, just replace  $\overrightarrow{\eta}^{\epsilon N}_s(0)$ by $\overleftarrow{\eta}^{\epsilon N}_s(N)$ and $\Theta_x^{-}$ by $\Theta_x^{+}.$ Indeed, if one goes through the previous proof one just has to note that  $\sum_{x \in \Lambda_N}r^{\pm}_N(\tfrac{x}{N})$ plays exactly the same role of the convergent sum $\tfrac{1}{\log(N)}\sum_{x \in\Lambda_N}\Theta_x^{\pm}$.
\end{rem}

\appendix
\section{Uniqueness of weak solutions}
\label{appendix}
In this section we prove uniqueness of the weak solutions introduced in Section \ref{HE}. We follow closely the proof that can be found in  Appendix 2 of \cite{KL}. We present the details of the proof in the Robin case and in the reaction-diffusion case, i.e. for the  { solutions in Definition \ref{Rob} and Definition \ref{reacdif}, since the uniqueness for the remaining weak solutions can be proved as  in  \cite{BGJO}. We note that in the reaction case we can  follow the proof in Appendix A of \cite{BGJO} since items (1) and (3) of Lemma A.1 also hold in our situation.}
\subsection{Robin case}

Consider two weak solutions $\rho^1$ and $\rho^2$ in the sense of Definition \ref{Rob} starting from the same initial condition,  and let $\bar \rho:=\rho^1-\rho^2$. Note that 
\begin{equation}\label{derivativepsi}
	\left\langle \bar \rho_{t},  G_{t} \right\rangle = \int_0^t\left\langle \bar \rho_{s},\Big(\partial_s +   c_2\Delta \Big) G_{s}  \right\rangle ds
\end{equation}
for any $G \in \mathcal{S}_{Rob}$. We consider the associated  {Sturm-Liouville} problem, as in  \cite{sturm}, which is given by
\begin{equation}\label{sturm}
	\begin{cases}
		& c_2 \psi''(u)+\lambda\psi(u)=0,  \quad u \in (0,1),\\
		&\psi'(0)=\tfrac{ m}{  c_2}\psi(0), \quad  \psi'(1)=-\tfrac{ m}{ c_2}\psi(1)\\
	\end{cases}
\end{equation}
for any $\lambda>0$. The solutions of $c_2 \psi''(u)+\lambda\psi(u)=0$ are of the form 
\begin{equation}
	\label{psi}
	\psi(u)=A \sin(\tfrac{\sqrt\lambda}{c_2} u) + B \cos(\tfrac{\sqrt\lambda}{c_2} u),
\end{equation}
for some constants $A$ and $B$. Let $\tilde \lambda :=\tfrac{\sqrt\lambda}{c_2}$ and imposing the condition $\psi'(0)=\tfrac{ m}{  c_2}\psi(0)$ we find that $$A=\frac{m}{c_2\tilde \lambda} B.$$
Then we apply the second boundary condition which is, thanks to the computations above,
\begin{equation}
	B\frac{m}{c_2}\cos(\tilde \lambda)-B\tilde\lambda \sin(\tilde \lambda) = -B\frac{m}{c_2}\Big(\frac{m}{\tilde\lambda c_2}\sin(\tilde \lambda)+ \cos(\tilde \lambda)\Big).
\end{equation}
The equation above coincide with the trascendental equation
\begin{equation}
	\tan (\tilde \lambda)=\frac{2mc_2 \tilde \lambda}{\tilde \lambda^2c_2^2-m^2}
\end{equation}
which has infinite solutions $\{\tilde \lambda_n\}_{n\geq 1}$.  {These} solutions behave exactly as described in \cite{sturm}, i.e. $0<\tilde \lambda_1<\tilde \lambda_2<\dots$ and for $n$ big enough they are of order $n$. So, since $\lambda_n=c_2^2\tilde \lambda_n^2$ we have that $0<\lambda_1<\lambda_2 \dots \lesssim \lambda_n \sim n^2 <\dots$. Moreover, as proved in \cite{sturm}, choosing, for any $n \geq 1$, a normalizing constant $B_n$, the set of functions $\{\psi_n\}_{n\geq 1}$ defined on $u \in [0,1]$ by
\begin{equation}\label{ortbas}
	\psi_n(u)=\frac{m}{\sqrt{\lambda_n}} B_n \sin\Big(\tfrac{\sqrt{\lambda_n}}{c_2} u\Big) + B_n \cos\Big(\tfrac{\sqrt{\lambda_n}}{c_2} u\Big),
\end{equation}
forms an orthonormal basis of $L^2$.

Now, we follow closely the strategy presented in Appendix  {2.4} of \cite{KL}. For any $t \in [0,T]$ we define \begin{equation}\label{Vt}V(t)=\sum_{n\geq1}\langle \bar \rho_t,\psi_n\rangle^2.\end{equation}
 {Observe that from \eqref{derivativepsi} we conclude that $V(t)$  is  differentiable in time. Therefore, }
the goal is  to show that $V'(t)\lesssim V(t)$ and from Gronwall's inequality we get  $V(t)\leq V(0)=0$.
The latter implies, by definition of $V$, that for any $t \in [0,T]$, $\bar \rho_t=0$ almost everywhere in $[0,1]$ which concludes the proof of uniqueness.  Let us then prove that $V'(t)\leq 0$. Fix a $t \in [0,T]$ and note that
\begin{equation}\label{Vder}
	V'(t)=2\sum_{n\geq 1}\langle \bar \rho_t,\psi_n\rangle\frac{d}{dt}\langle \bar \rho_t,\psi_n\rangle.
\end{equation}

From \eqref{derivativepsi} we see that
\begin{equation}
	\frac{d}{dt}\langle \bar \rho_t,\psi_n\rangle=\frac{d}{dt}\int_0^t \langle \bar \rho_s, c_2\psi_n''\rangle ds = \langle \bar \rho_t,\psi''_n\rangle= - \frac{\lambda_n}{c_2}\langle \bar \rho_t, \psi_n\rangle
\end{equation}
where last equality comes from the definition of $\psi_n.$ From this we get
\begin{equation}
	V'(t)=-2\sum_{n\geq 1}\frac{\lambda_n}{c_2}\langle \bar \rho_t,\psi_n\rangle^2 \leq 0 < V(t).
\end{equation}
The first inequality above is a consequence of  the fact that the eigenvalues $\{\lambda_n\}_{n\geq 1}$ are positive. This ends the proof.

\subsection{Reaction-diffusion case}

 This proof goes under the same strategy above and  the analysis of the associated  Sturm-Liouville problem as  in Section 5.3 of \cite{fluc}.
As above, for any $t \in [0,T]$, let $\bar \rho_t=\rho^1_t-\rho^2_t$, where $\rho^1$ and $\rho^2$ are two weak solutions of the reaction-diffusion equation with the same initial condition in the sense of Definition \ref{reacdif}. Then,  the associated Sturm-Liouville problem is given by
\begin{equation}\label{sturmreac}
	\begin{cases}
		& c_2\psi''(u)=\kappa V_1(u)\psi(u)-\lambda\psi(u),  \quad u \in (0,1),\\
		& \psi(0)=0, \quad  \psi(1)=0,\\
	\end{cases}
\end{equation}
for some $\lambda>0$. In Proposition 5.9 of \cite{fluc} it is proved that this problem admits solutions $\{\lambda_n\}_{n\geq 1}$ such that
$0<\lambda_1<\lambda_2<\lambda_3<\dots$ which behave, for $n$ big enough as $\lambda_n\gtrsim n^2$. Moreover, the solutions $\{\psi\}_{n\geq 1}$ of the problem above for $\lambda$ taken in the set  $\{\lambda_n\}_{n\geq 1}$ form an orthonormal basis  of $L^2$.

As above, the proof ends by noting that
\begin{equation}
	\frac{d}{dt}\langle \bar \rho_t,\psi_n\rangle=\frac{d}{dt}\int_0^t \langle \bar \rho_s, c_2\psi_n''-\kappa V_1\psi_n\rangle ds = - \lambda_n\langle \bar \rho_t, \psi_n\rangle
\end{equation}
where the first equality comes from the definition of $\bar\rho$  and the second from \eqref{sturmreac}.

\vspace{1cm}
\thanks{ {\bf{Acknowledgements:\\ }}
\quad

P.G. and S.S. thank  FCT/Portugal for support through the project 
UID/MAT/04459/2013.  This project has received funding from the European Research Council (ERC) under  the European Union's Horizon 2020 research and innovative programme (grant agreement   n. 715734).}

\nocite{*}
\bibliography{bibliografia}
\bibliographystyle{plain}

\end{document}